\documentclass[11pt]{amsart}

\usepackage{amscd}
\usepackage{amsthm}
\usepackage[centertags]{amsmath}
\usepackage{amsfonts}
\usepackage{newlfont}
\usepackage{graphicx}
\usepackage[usenames]{xcolor}
\usepackage{amsfonts, amssymb,accents}
\usepackage{mathrsfs}
\usepackage{latexsym}
\usepackage{bbold}
\usepackage{enumitem}
\usepackage{tikz-cd}

\usepackage{hyperref}

\newtheorem{theorem}{Theorem}[section]

\newtheorem{proposition}[theorem]{Proposition}
\newtheorem{corollary}[theorem]{Corollary}
\newtheorem{lemma}[theorem]{Lemma}
\newtheorem{remark}[theorem]{Remark}
\newtheorem{definition}[theorem]{Definition}
\newtheorem{example}[theorem]{Example}

\newtheorem{question}[theorem]{Question}

\newcommand{\PP}{\mathscr P}

\newcommand{\zN}{\mathbb N}

\newcommand{\sub}{\subseteq}
\newcommand{\zT}{\mathbb T}

\newcommand{\f}{\frac}

\newcommand{\AP}{\mathcal {AP}}
\newcommand{\kAP}{\mathcal {AP}\!_b}
\newcommand{\A}{\mathcal {F}}
\newcommand{\PS}{\mathcal {PS}}

\newcommand{\BD}{\overline{\mathcal {BD}}}
\newcommand{\F}{\mathcal {F}}

\renewcommand{\l}{\left(}
\renewcommand{\r}{\right)}

\title{Frequently recurrence properties and block families}
\keywords{Frequently recurrent operators, Reiteratively hypercylic operators, hereditary upward families}
\subjclass[2010]{
47A16, 
    37B20  	
	37A45  
		47B37  	
}

\author{Rodrigo Cardeccia, Santiago Muro}
\address{Instituto Balseiro, Universidad Nacional de Cuyo – C.N.E.A. and CONICET, San Carlos de Bariloche, Rep\'ublica Argentina} \email{rodrigo.cardeccia@ib.edu.ar} 
\address{FCEIA, UNIVERSIDAD NACIONAL DE ROSARIO AND   CIFASIS, CONICET}
\email{muro@cifasis-conicet.gov.ar}

\thanks{Partially supported by ANPCyT PICT 2018-04250, UBACyT 20020130300052BA, PIP 11220130100329CO and CONICET}

\begin{document}

\begin{abstract}
We prove that reiteratively hypercyclic operators have perfect spectrum. Consequently, it follows that there exist separable infinite dimensional Banach spaces that do not support any reiteratively hypercyclic
operator.
For this, we study $\F$-recurrence and almost $\mathcal {F}$-recurrence of operators for general families and in particular for a  special class of families, called \emph{block} families.
\end{abstract}

\maketitle

\section{Introduction and main results}

Reiterative hypercyclicity is a strengthening of hypercyclicity, which is the main object
of study in linear dynamics—an emerging field at the intersection of operator theory and
functional analysis, closely tied to the invariant subspace problem \cite{BayMat09,beauzamy1988introduction}. Within
this framework, there exists a strict hierarchy of linear dynamical notions, including frequent
hypercyclicity or chaoticity, which implies reiterative hypercyclicity, and reiterative hypercyclicity, in turn, implies weak mixing and hence hypercyclicity \cite{Bes16}. It is a known
fact that any infinite-dimensional separable Banach space supports a weakly mixing operator \cite{ansari1997existence,bernal1999hypercyclic,grivaux2005hypercyclic}.

In 2011, Argyros and Haydon \cite{ArgHay11} (developing the techniques for the construction of Hereditarily Indecomposable spaces  \cite{GowMau93}) constructed a Banach space with the peculiar property of supporting very
few operators. Specifically, every operator on this space can be expressed as a sum of a scalar
multiple of the identity plus a compact operator. This construction is notable for various
properties it possesses. Indeed, Lomonosov’s famous theorem shows that every operator
that commutes with a non-zero compact operator must have a non-trivial closed invariant
subspace. As a consequence, every continuous linear operator on the Argyros-Haydon space
is guaranteed to have a non-trivial closed invariant subspace, making it a groundbreaking
discovery in the field of operator theory. As a matter of fact, no other infinite-dimensional
separable Banach space with this property is currently known besides the one constructed
by Argyros and Haydon. So the Argyros-Haydon space remains a unique and significant
example in this regard.

It has been established that neither a chaotic nor a frequently
hypercyclic operator can be supported by the Argyros-Haydon space \cite{BonMarPer01,Shk09}.

The concept of reiterative hypercyclicity is inspired in that of frequent hypercyclicity but it shares some properties with hypercyclic operators, e.g. for a reiteratively hypercyclic operator (in contrast to a frequently hypercyclic one) the set of reiteratively hypercyclic vectors is residual, and every hypercyclic vector is automatically reiteratively hypercyclic. It is then natural to ask 
the following question (see e.g. \cite[Question 4.10]{BonGroLopPer22JFA} or \cite[Question 2.6]{MarPui21}).
\begin{question}\label{pregunta reiterative}
Does every infinite dimensional separable Banach space support reiteratively hypercyclic operators?
\end{question}

The primary purpose of this note is to establish an additional property of the Argyros-Haydon space, namely, its failure to support any reiteratively hypercyclic operator, in particular answering Question \ref{pregunta reiterative}.
\begin{theorem}\label{teo:argyros}
 The Argyros-Haydon space does not support a reiteratively hypercyclic operator.    
\end{theorem}
In relation to this theorem it would be interesting to know if a Banach space without frequently hypercyclic operators could support a reiteratively hypercyclic operator (see Question \ref{pregunta existe reiterative no existe frequente}).

We obtain our main result as a consequence of our investigations on the relationship between the concept of almost $\mathcal {F}$-recurrence and $\F$-hypercyclicity.
Recurrence is one of the oldest and most studied concepts in the theory of dynamical systems.  In contrast, in linear dynamics the main concept is that of hypercyclicity and a systematic study of recurrence began only around 10 years ago with the work of Costakis, Manoussos and Parissis \cite{CosManPar14,CosPar12}.  

On the other hand, based on the seminal articles \cite{Bes16,BonGro18} much of the interest in linear dynamics shifted towards the study of $\F$-hypercyclicity for different hereditarily upward (or Furstenberg) families of subsets of the natural numbers, hence including in a single theory concepts as hypercyclity, frequent hypercyclicity, reiterative hypercyclicity or chaos (see also \cite{CardeMur22}), among others.

At this point, it turns out to be natural to investigate $\F$-recurrence of linear operators for general families $\F$, a task recently initiated by Bonilla, Grosse-Erdmann, L\'opez-Mart\'inez and Peris \cite{BonGroLopPer22JFA}.  One of the main goals in \cite{BonGroLopPer22JFA} was to determine conditions under which an operator that is hypercyclic and $\F$-recurrent is necessarily $\F$-hypercyclic. Our point of view is slightly different, and a question we ask is whether hypercyclicity together with a strong property of recurrence ($\F$-recurrence) implies a weak property of frequent hypercyclicity ($\mathcal G$-hypercyclicity, for a  family $\mathcal G$ that is bigger than $\F$). It is worth mentioning that hypercyclicity together with the strongest possible form of recurrence, density of periodic points, does not necessarily imply frequently hypercyclicity. This was proved by Menet in \cite{Men17} who, answering one of the main open questions at that moment in linear dynamics, exhibited an example of an operator that is chaotic but not upper frequently hypercyclic.

With this aim in mind we study the block family $b\F$ associated to a given family $\F$, which is essentially formed by sets that are union of finite blocks of a given set $A\in \F$. We also introduce the notion of almost $\F$-recurrence and investigate the relationship between $b\F$-hypercyclicity and almost $\F$-recurrence plus hypercyclicity.

An application of our results on almost $\F$-recurrence to the family of sets of positive upper Banach density 
leads us to the following theorem which implies our main result.
\begin{theorem}\label{teo: perturbaciones identidad no reiter} 
On real or complex locally convex spaces, no compact perturbation of the identity can be reiteratively hypercyclic.
\end{theorem}

Almost $\mathcal {F}$-recurrence is also important for the understanding of the dynamics 
of backward shift operators: in \cite{BonGro18}, it was proved that, when the family $\F$ has the property of being \emph{upper}, $\F$-hypercyclic unilateral backward shifts are characterized by two conditions. We show in Section \ref{section shifts} that exactly the same two conditions  characterize almost $\mathcal {F}$-recurrent unilateral backward shifts for general, non necessarily upper, families. We also prove that under the assumption of $\F$ being an upper family, every  almost $\mathcal {F}$-recurrent unilateral backward shift is necessarily $\F$-hypercyclic. Moreover, we also show that two very similar conditions characterize almost $\mathcal {F}$-recurrent bilateral backward shifts, but in this case it is  unclear if they are also $\F$-hypercyclic.

The paper is organized as follows. Section \ref{section spectrum} is dedicated to the study of the spectrum
of reiteratively hypercyclic operators. We establish that the spectrum of a reiteratively
hypercyclic operator is perfect (Theorem  \ref{P_bd perfecto}). Consequently, Theorem \ref{teo:argyros} and Theorem \ref{teo: perturbaciones identidad no reiter} follow.
The proof of Theorem \ref{P_bd perfecto} relies on the notion of almost $\mathcal {F}$-recurrence.

In Section \ref{section block}, we delve into the concept of almost $\F$-recurrence for general families. Among
other results, we demonstrate that for block families, if an operator is hypercyclic and almost
$\F$-recurrent, then it is $b\F$-hypercyclic (Proposition \ref{equivalencia bF hiper}). Additionally, we establish a converse
of this result for adjoint operators: if an adjoint operator is almost $b\F$-recurrent, then it is
also almost $\F$-recurrent (Theorem \ref{PBF implica PF}). {We also establish that for CuSP families,} almost
$\F$-recurrence remains invariant under unimodular scalar multiplications (Theorem \ref{lambda T CuSP}).

In section \ref{section backwardshifts} we study almost $\mathcal {F}$-recurrence for backward shifts. We give a characterization of almost $\mathcal {F}$-recurrent unilateral and bilateral backward shifts. We also show several examples and counterexamples.

\section{Spectrum of reiteratively hypercyclic operators}\label{section spectrum}

In this section we prove the results stated in the Introduction and, in particular that 
there are infinite dimensional separable Banach spaces not supporting reiteratively hypercyclic operators (Theorem \ref{teo:argyros}). These results will be deduced from Theorem \ref{teo quasinil} below.
The proofs of some preliminary results are delayed to the next section. Unless explicitly specified, all the results hold for both complex and real Banach spaces.

Let us first recall some definitions. 

The \emph{upper Banach density} of a subset $A\subset \zN_0=\zN\cup\{0\}$ is $\overline{Bd}(A):=\lim_n \sup_k \f{\# (A\cap [k,k+n])}{n+1}$. The family of subsets of $\zN_0$ with positive upper Banach density is denoted by $\overline {\mathcal{BD}}$. 
The \emph{lower density} of a subset $A\subset \zN_0$ is $\underline {d} (A):=$ $\liminf_n  \f{\#(A\cap[0,n])}{n+1}$). The family of subsets of $\zN_0$ with positive lower density is denoted by $\underline {\mathcal{D}}$.

An operator $T$ on a locally convex space $X$ is \emph{reiteratively hypercyclic} (see \cite{Bes16}) if there is some $x\in X$ such that for any nonempty open set $U,$ the set of return times $N_T(x,U):=\{n\ge1\,:\, T^n(x)\in U\}$ (or just $N(x,U)$ if the operator is clear from the context) has positive upper Banach density.   

An operator $T$ is \emph{recurrent} if for every nonempty open set $U$ there is $x\in U$ such that $N(x,U)\ne \emptyset.$ 
A vector $x$ is \emph{recurrent} for $T$ if for every nonempty open set $U$ containing $x,$ the set of return times $N(x,U)\ne \emptyset.$ By \cite[Proposition 2.1]{CosPar12}, an operator is recurrent if and only if it has a dense set of recurrent vector.

A vector $x$ is \emph{reiteratively recurrent} (see \cite{BonGroLopPer22JFA}) for $T$ if for every nonempty open set $U$ containing $x,$ the set of return times $N(x,U)$ has  positive upper Banach  density. The operator $T$  is \emph{reiteratively recurrent} if the set of reiteratively recurrent vectors is dense in $X.$ 

We also need to introduce the following notion of recurrence, which is slightly  weaker than reiterative recurrence (see also Definition \ref{definition P_F}).
\begin{definition}
Let $T$ be a linear operator. We say that $T$ is \emph{almost reiteratively recurrent}, or almost $ \overline {\mathcal{BD}}$-recurrent,  provided that for every nonempty open set $U$, there is $x\in U$ such that $N(x,U)\in \overline {\mathcal{BD}}$.
\end{definition}
Note that by \cite[Proposition 2.1]{CosPar12}, almost reiteratively recurrent operators are recurrent.

A map $S:Y\to Y$ is said to be a factor of $T:X\to X$ if there exists a continuous and surjective map $J : X \to Y$ such that  $JT = SJ$. It is easy to see that if $T$ is almost reiteratively recurrent then any factor of $T$ is.
Finally, recall that an operator $S$ is said to be quasinilpotent provided that $\|S^n\|^{1/n}\rightarrow 0$ as $n\to \infty$.

\begin{theorem}\label{teo quasinil}
Let $X$ be a Banach space, $T$ be a linear operator on $X$ such that $T$ is almost reiteratively recurrent and  $\lambda\in\mathbb C$ with $|\lambda|=1$. Then either $T=\lambda Id_X$ or $T-\lambda Id_X$ is not quasinilpotent.
\end{theorem}

For the proof of Theorem \ref{teo quasinil} 
we need three preliminary results, the first 
is a lemma which was proved in \cite[Lemma 3.21]{CardeMur22}. Its proof is a modification of a result of Shkarin for frequently hypercyclic operators (see \cite{Shk09} or \cite[Lemma 9.38]{GroPer11}). 
\begin{lemma}\label{lema qNIL}
	Let $X$ be a  Banach space. Let $S$ be a linear operator on $X$, $x^*\in X^*$ and $U=\{y\in X: Re(\left\langle y,x^*\right\rangle)>0,Re(\left\langle S(y),x^*\right\rangle)<0\}$. If for some $x\in U$, \begin{equation*}\label{eq lema espectro SALAP}
	\underline{d}\big(N_S(x,U)\big)>0.    
	\end{equation*}
	Then $S-Id_X$ is not quasinilpotent.
	\end{lemma}

The other two preliminary results 
are immediate consequences of the more general Theorem \ref{lambda T CuSP} and Theorem \ref{reiterative recurrent implies freq T**}, which  will be proved in Section  \ref{section block families}. 
\begin{proposition}\label{lambda T reiterative}
    If $T$ is almost reiteratively recurrent then for any $|\lambda|=1$, $\lambda T$ is almost reiteratively recurrent.
\end{proposition}
In the following we will identify points of $X$ with their image in the bidual $X^{**}$ through the canonical inclusion $X\hookrightarrow X^{**}$.
\begin{theorem}\label{reiterative recurrent implies freq T** BD}
If $T$ is almost reiteratively recurrent then for each nonempty open set of the bidual $V\sub X^{**}$ such that $V\cap X\ne\emptyset$, there is $z\in V$ such that 
 $N_{T^{**}}(z,V)\in \underline {\mathcal{D}}$. 
\end{theorem}

We are now able to present the:
\begin{proof}[Proof of Theorem \ref{teo quasinil}]
Let $T$ be almost reiteratively recurrent and $|\lambda|=1$.
Suppose that $T\neq \lambda Id_X$. We have to prove that $T-\lambda Id_X$ is not quasinilpotent.
By Proposition \ref{lambda T reiterative} we may assume that $\lambda=1$. 
Using basic properties of adjoint operators, we have that $\|(T-Id_X)^n\|=
\|((T-Id_X)^n)^{**}\|=\|(T^{**} - Id_{X**})^n\|.$ As a consequence, $T -Id_X$ is quasinilpotent if and only if $T^{**}-Id_{X^{**}}$ is quasinilpotent. Thus for our purposes, it is sufficient to prove that $T^{**}-Id_{X^{**}}$ is not quasinilpotent.

Since $T\neq Id_X$, there is $x\in X$ such that $x\neq T(x)$ and hence by the Hahn-Banach Separation Theorem there is $x^*\in X^*$ such that $Re(\left\langle x,x^*\right\rangle)>0$ and $Re(\left\langle T^{**}(x),x^*\right\rangle)<0$.

 Consider $U:=\{y\in X^{**}: Re(\left\langle y,x^*\right\rangle)>0,Re(\left\langle T^{**}(y),x^*\right\rangle)<0\}$. Since $x\in X\cap U$, there is, by Theorem \ref{reiterative recurrent implies freq T** BD}, $x^{**}\in U$ such that $N_{T^{**}}(x^{**},U)$ has positive lower density. It follows by Lemma \ref{lema qNIL} that $T^{**}- Id_{X^{**}}$ is not quasinilpotent.
\end{proof}

We finish this section showing the main consequences of Theorem  \ref{teo quasinil}.
\begin{theorem}\label{P_bd perfecto}
Let $X$ be a complex Banach space. If $T$ is almost reiteratively recurrent then either $T$ has perfect spectrum (i.e. the spectrum doesn't have isolated points) or $T$ has a factor that is a unimodular multiple of the identity. 
\end{theorem}
 \begin{proof}
Suppose that $T$ is almost reiteratively recurrent and that $\sigma(T)$ (the spectrum of $T$) is not perfect. Then $\sigma(T)$ has an isolated point $\lambda$. By the Riesz Decomposition Theorem (see e.g. \cite[Theorem B.9]{GroPer11}) there are invariant closed subspaces $M_1,M_2$ such that $X=X_1\oplus X_2$ and $\sigma(T|_{M_1})= \sigma(T)\setminus \{\lambda\}$, $\sigma(T|_{M_2})= \{\lambda\}.$  Consider $\tilde T=T|_{M_2}:M_2\to M_2$. Then $\tilde T$ is a factor of $T$ with $J:X\to M_2$ the projection from $X$ onto $M_2$. Therefore $\tilde T$ is almost reiteratively recurrent.  
In particular,
$\tilde T$ is recurrent. It follows by \cite[Proposition 2.11]{CosManPar14} that $\lambda\in \zT$. Since $\sigma(\tilde T)=\{\lambda\}$, then $\sigma(\tilde T-\lambda Id_{M_2})=\{0\}$ and since we are on a complex Banach space, this is equivalent to  $\tilde T-\lambda Id_{M_2}$ being quasinilpotent.  Applying    Theorem \ref{teo quasinil}  we obtain that $\tilde T=\lambda Id_{M_2}$.
\end{proof}
Since hypercyclic operators cannot have multiples of the identity as factors (recall that every factor of a hypercyclic operator is also hypercyclic, see e.g. \cite[Proposition 2.24]{GroPer11}),
Theorem \ref{P_bd perfecto} applied to reiteratively hypercyclic operators gives the following.
\begin{corollary}
    \label{bd perfecto}
Reiteratively hypercyclic operators on complex Banach spaces  have perfect spectrum. 
\end{corollary}

From the above theorem we can deduce that no compact perturbation of the identity is reiteratively hypercyclic and that there are infinite dimensional and separable Banach spaces without reiteratively hypercyclic operators.
\begin{proof}[Proof of Theorem \ref{teo: perturbaciones identidad no reiter}]
If $X$ is a complex Banach space, then the result follows immediately from Corollary \ref{bd perfecto}.

For real Banach spaces, let us assume that $T=\lambda Id_X+K$, where $K$ is a compact operator and $T$ is reiteratively hypercyclic. 
Consider $X_\mathbb C$, the complexification of $X$. Thus, if $\tilde T$ and $\tilde K$ denote the complexificactions of $T$ and $K$ respectively, they satisfy that $\tilde T=Id_{X_\mathbb C}+\tilde K$ and that $\tilde K$ is a compact operator. Since $\tilde T$ is a factor of $T\oplus T$ (with $J$ given by the canonical  isomorphism between $X\oplus X$ and $X_\mathbb C$) and $T\oplus T$ is reiteratively hypercyclic (See \cite[Corollary 2.8]{ErnEssMen21} or Corollary \ref{cor: prod F-hiper}), it follows that $\tilde T$ is reiteratively hypercyclic, which is a contradiction.

If $X$ is a locally convex space and $T=\lambda Id_X+K$ with $K$ a compact operator, then there  exists a Banach space $Y$ and a linear operator $S:Y\to Y$, with $S=Id_Y+K'$ and $K'$ compact such that $S$ is a factor of $T$ (see for example the proof of \cite[Proposition 6.1]{MarPer02} or 
 \cite[Ex. 12.2.3]{GroPer11}). Thus $S$, and hence $T$, cannot be reiteratively hypercyclic. 
\end{proof}

\begin{proof}[Proof of Theorem  \ref{teo:argyros}]
Since every linear operator on the Argyros-Haydon space is a compact perturbation of a multiple of the identity 
 \cite{ArgHay11}) we have, by Theorem \ref{teo: perturbaciones identidad no reiter},  that no operator on it can be reiteratively hypercyclic.
\end{proof}

\section{Block families,  $\F$-hypercyclicity and almost $\mathcal {F}$-recurrence}\label{section block families}\label{section block}

In this section we study hypercyclicity and almost $\F$-recurrence for general families. We prove general results that have Proposition \ref{lambda T reiterative} and Theorem \ref{reiterative recurrent implies freq T** BD} as particular cases.
 
 \subsection{Preliminaries}
  A family of subsets of the natural numbers $\F\sub \mathcal P(\zN_0)$ is said to be \textit {hereditary upward}  provided that if $A\in \F$ and $A\sub B$, then $B$ belongs also to $\F$. Hereditary upward families are also sometimes called \textit{Furstenberg families}. Given an hereditary upward family, we say that an operator $T$ on a space $X$ is  \emph{$\F$-hypercyclic} if there is $x\in X$ for which the set $N_T(x,U):=\{n\in\zN_0: T^n(x)\in U\}$ of return times belong to $\F$, for every nonempty open set $U\subset X$.
 A vector $x\in X$ is  \emph{$\F$-recurrent} for $T$ if $N_T(x,V)\in\F$ for every neighborhood $V$ of $x$. An operator is called \emph{$\F$-recurrent} if the set of $\F$-recurrent vectors is dense. The set of return times $N_T(x,U)$ will be denoted by $N(x,U)$ when the operator is clear from the context.

 Thus, for example, if we denote by $\A_\infty$  the family of infinite sets, $\A_\infty$-hypercyclicity and $\F_\infty$-recurrence are simply hypercyclity and recurrence respectively. 
 
Let us recall some examples of families, which are the most widely studied in the literature. The families $\overline{\mathcal{BD}}$ and $\underline{\mathcal{D}}$ were already recalled in the previous section.
\begin{itemize}
\item $\overline{\mathcal{BD}}=\{$ sets with positive Banach upper density$\}$.
An operator is \emph{reiteratively hypercyclic} (resp. reiteratively recurrent) if and only if it is $ \overline{\mathcal{BD}}$- hypercyclic (resp. $ \overline{\mathcal{BD}}$-recurrent). 

\item $\underline{\mathcal{D}}=\{$ sets with positive lower density$\}$.
An operator is \emph{frequently hypercyclic} if and only if it is $ \underline{\mathcal{D}}$- hypercyclic. 

\item $\overline{\mathcal{D}}=\{$ sets with positive upper density$\}$, i.e. $A\in  \overline{\mathcal{D}}$ if $\overline {d} (A):=$ $\limsup_n  \f{\# (A\cap [0,n])}{n+1}>0$.     An operator is \emph{upper frequently hypercyclic} if and only if it is $ \overline{\mathcal{D}}$- hypercyclic.
If $\overline {d} (A)=\underline {d} (A)=d$ then $A$ is said to have density $d$. The family of sets with positive density is denoted by  $\mathcal{D}.$

\item
$\overline{\mathcal{LD}}= \{ \text{sets with positive upper logarithmic density}\}$, i.e $A\in  \overline{\mathcal{D}}$ if  $\overline \delta (A):=\limsup_N \f{1}{\log N} \sum_{1\leq j\leq N, j\in A} \f{1}{j}>0$. The families  $\underline{\mathcal{LD}}$ and ${\mathcal{LD}}$ of sets with positive lower logarithmic density and positive logarithmic density are defined analogously. 

Given a set of natural numbers $A$, the following relation holds (see e.g. \cite[Part III, Theorem 1.2]{Ten15}):
\begin{equation}\label{log density vs density}
\underline{Bd}(A)\leq  \underline{d}(A)\leq \underline {\delta}(A)\leq \overline {\delta}(A)\leq \overline {d} (A) \leq \overline {Bd} (A).
  \end{equation}

Thus, if a set has density, then it has logarithmic density.

\item $ Thick=\{ \text{ sets with arbitrarily long intervals }\}$. In
other words, a set $A$ is thick if for every $L\ge 1$, there exists $a \ge 0$ such that $[a, a + L] \subset A$. 
\item $\mathcal{S}yn=\{ \text{ Syndetic sets, i.e. they have bounded gaps }\}$. Equivalently, $A\in\mathcal Syn$ if there is a finite subset $F$ of $S$ such that $\bigcup_{n\in F} S-n\supseteq\zN$.
\item $\mathcal {PS}yn=\{$ Piecewise syndetic sets, i.e. there is $b$ such that $A$ supports arbitrarily long subsets with gaps bounded by $b$ $\}$, where a set  $A$ is piecewise syndetic if there exists
a thick set B and a syndetic set $C$ such that $A = B\cap C$. Equivalently, $A\in \mathcal{PS}yn$ if there is a finite set $F$ such that 
$$ \bigcup_{n\in F} A+n \text{ is thick.}$$

    \item $\kAP$=$\{$ sets with arbitrarily long arithmetic progressions with bounded step$\}$. That is, a set $A$ belongs to $\kAP$ if there is some $k>0$ such for any $m\in\mathbb N$ there exists an arithmetic progression of length $m+1$ and step $k\in\mathbb N$, i.e. for some $a\in\mathbb N$ such that $\{a,a+k,a+2k,\dots,a+mk\}\subset A$.

\end{itemize}

An hereditary upward family $\F$ is said to be \textit {shift invariant} (right-shift invariant or left-shift invariant) if it satisfies that $A\in\F$ if and only if $(A+k)\cap\mathbb N_0\in\F$ for any $k\in\mathbb Z$ ($k\in\mathbb N_0$ or $-k\in\mathbb N_0$) and it is said to be \textit{finitely invariant} provided that for every $A\in \F$ and every $k\in\mathbb N$, we have that $[k,\infty)\cap A\in \F$. All families mentioned above are shift invariant and finitely invariant.

An hereditary upward family is said to be \textit {upper} provided that $\emptyset\notin \A$ and $\A$ can be written as 
\begin{align}\label{upper}
\bigcup_{\delta\in D} \A_\delta,\quad\textrm{with }\; \A_\delta=  \bigcap_{m\in M} \A_{\delta,m},    
\end{align}
where $D$ is arbitrary but $M$ is countable and  such that the families $\A_{\delta,m} $ and $\A$ satisfy
\begin{itemize}
\item each $\A_{\delta,m}$ is \textit{finitely hereditary upward}, that means that for each $A\in \A_{\delta,m}$, there is a finite set $F$ such that $F\cap A\sub B,$ then $B\in \A_{\delta,m}$;
\item $\A$ is \textit{uniformly left invariant}, that is, if $A\in\A$ then there is $\delta\in D$ such that for every $n\in \mathbb N_0$, $A-n\in \A_\delta$. 
\end{itemize}

Let $\F =\bigcup_{\delta\in D} \F_\delta$ with $\F_\delta=\bigcap_{\mu\in M} \F_{\delta,\mu}$ be an upper Furstenberg family. Then it is called \emph{uniformly finitely invariant} if, for any $F\in \F$, there is some $\delta\in D$ such that, for all $n\geq 0$, $F\setminus[0,n]\in \F_\delta$.

All the upper families considered so far are uniformly finitely invariant, see \cite[Example 2.11] {BonGro18}.  
The next theorem provides some useful equivalences for $\F$-hypercyclicity for upper families. 
\begin{theorem}[Bonilla-Grosse Erdmann \cite{BonGro18}]\label{equivalencias upper}
Let $\A$ be an upper hereditary upward family and let $T$ be a linear operator on a separable Fr\'echet space. Then the following are equivalent:
\begin{enumerate}
\item $T$ is $\A$-hypercyclic.
    \item For any nonempty open set $V$ there is $\delta\in D$ such that for any nonempty open set $U$ there is $x\in U$ with $N_T(x,V)\in \A_\delta.$
    \item For any nonempty open set $V$ there is $\delta\in D$ such that for every $U$ and $m\in M$ there is $x\in U$  with $N_T(x,V)\in \A_{\delta,m}$.
    \item The set of $\A$-hypercyclic points is residual.
\end{enumerate}

\end{theorem}

 The families $\A_{\infty}$, $\overline{\mathcal D},\overline{\mathcal BD}$, $\overline{\mathcal{LD}}$ are upper while $ \underline{\mathcal{D}}$ is not upper (see \cite{BonGro18}).
The family $\underline{\mathcal{LD}}$ is not upper. This follows as the case $ \underline{\mathcal{D}}$: with essentially the same proof as \cite[Theorem 6.25]{BayMat09} it holds that if $T$ is an operator such that $T^n(x)\to 0$ for all $x$ in a dense set, then the set of $\underline{\mathcal{LD}}$-hypercyclic vectors is of first category. Then, if we take for example $T=2B$ on $\ell_2$, $T$ is frequently hypercyclic and thus $\underline{\mathcal{LD}}$-hypercyclic, and clearly $T^n(x)\to 0$ for all $x$ in a dense set. Therefore the set of $\underline{\mathcal{LD}}$-hypercyclic vectors of $T$ is not residual. Theorem \ref{equivalencias upper}  thus implies that the family  $\underline{\mathcal{LD}}$ is not upper.

For basic concepts in linear dynamics we refer to the books \cite{BayMat09,GroPer11}. For more on $\A$-hypercyclicity and recurrence see \cite{Bes16,BonGro18, BonGroLopPer22JFA, CardeMur22,ErnEssMen21,GriLop23,grivaux2022questions,lopez2022recurrentarxiv,Pui17,Pui18,Shk09}).

\subsection{Almost $\mathcal {F}$-recurrence and block families}

We define now a weak notion of $\mathcal {F}$-recurrence. 
\begin{definition}\rm\label{definition P_F}
Let $T$ be a linear operator and let $\F$ be an hereditary upward family. We will say that $T$  is \emph{almost} $\F$-\emph{recurrent} provided that for every nonempty open set $U$, there is $x\in U$ such that $N(x,U)\in \F$.
\end{definition}
Of course,   $\F$-\emph{recurrence} implies almost $\F$-\emph{recurrence}, we don't know if the converse is true in general, see Section \ref{final comments}.

For the family $\F_\infty$ of infinite sets almost $\mathcal {F}_\infty$-recurrence is exactly the notion of recurrence as defined classically (see e.g. \cite{Fur81}) and it is equivalent to $\F_\infty$-recurrence  \cite[Proposition 2.1]{CosManPar14}.
For the family $\underline{ \mathcal D}$ of sets with positive lower density, almost $\underline{ \mathcal D}$-recurrence was considered in \cite[Section 2.5]{GriMat14}. 

Almost $\mathcal F$-recurrence is slightly stronger than the $\PP_\F$ property as defined in \cite{Pui18}, and both concepts coincide for  shift invariant families. Indeed, the only difference is that for almost $\mathcal F$-recurrence we impose the vector $x$ to be in the open set $U$.  For the family of sets with positive upper density this property was also considered, with different names, in \cite{CosPar12} and  \cite[Proposition 4.6]{badea2007unimodular}.

We will be specially interested in almost $\F$-recurrence for a special class of families called block families (see  \cite{glasner2004classifying}, see also \cite{huang2012family,Li11}). 

\begin{definition}
Given an hereditary upward family $\F$ we define the block family $b\F$ in the following way: $S \in b\F$ if and only if there is $F\in \F$ such that for every finite subset $R$ of $F$ there is $n\in\zN_0$ such that $R+n\sub S$.
\end{definition}

\begin{proposition}[Basic properties]
Let $\F,\tilde \F$ be hereditary upward families, then 
\begin{enumerate}
    \item $\F\sub b\F$;
    \item $bb\F=b\F$;
    \item If $\F\sub \tilde \F$ then $b\F\sub b\tilde \F$;
    \end{enumerate}
\end{proposition}
\begin{proof}
(1) and (3) are immediate. Proof of (2). By (1) we have that $b\F\sub bb\F$ so that one inclusion holds.

Let $S\in bbF$, hence there is $F_b\in b\F$ such that for every finite subset $R_b$ of $F_b$ there is $n$ such that $R_b+n\sub S$. Also, since $F_b\in b\F$, there is $F\in \F$ such that for every finite subset $R$ of $F$ there is $n$ such that $R+n\sub F_b$. 
Therefore, if $R$ is a finite subset of $F$ and $R+n\sub F_b$, there is $n'$ such that $R+n+n'\sub S$, which proves that $S\in b\F$.

 \end{proof}
We provide next some examples of block families.
\begin{example}\label{ejemplos block families}
\begin{enumerate}
    \item [a)] Let $\F_{\infty}$ and $\mathcal {C}ofin$, where   $\mathcal {C}ofin$ denotes the family of cofinite subsets of $\mathbb N_0$.  Then, $b\F_{\infty}=\F_{\infty}$ and $b\mathcal{C}ofin=Thick$. 
    In particular, $bThick=Thick$.
    \item  [b)]$b\mathcal D$=$b\underline {\mathcal {D}}=b\overline{\mathcal D} =\overline{\mathcal{BD}}$. In particular, $b\overline{\mathcal{BD}}$=$bb\overline{\mathcal D}=b\overline{\mathcal D} =\overline{\mathcal{BD}}$.
    \begin{proof}
    One inclusion is Ellis'
    Theorem \cite[Theorem 3.20]{Fur81}, which says precisely that if $A\in \mathcal{BD}$, then there is $F$ of positive density such that for every finite subset $R$ of $F$ there is $n\in \mathbb N_0$ with $R+n\sub A$. Thus, $\overline{\mathcal {BD}}\sub b\mathcal D$. 
    
    For the converse,  let $A\in b\overline{ \mathcal {D}}$. Hence there is $F\in b\overline{ \mathcal {D}}$ such that for every finite subset $R$ of $F$ there is $n\in\zN_0$ such that $R+n\sub A$. As $F$ has positive upper density, there are $\delta>0$ and $(n_k)_k$ such that   $\f{\#(F\cap[0,n_k])}{n_k+1}>\delta$ for every $k\in \mathbb N$. 
    
    Let $R_k=[0, n_k]\cap F\sub F$, hence $\f{\# (R_k\cap[0,n_k])}{n_k+1}>\delta$ and there is $m_k$ such that $R_k+m_k\sub A$. It follows that $\f{\#(A\cap [m_k,m_k+n_k])}{n_k+1}\geq \f{\#(R_k\cap [0,n_k])}{n_k+1}>\delta$. We conclude that $A$ has positive upper Banach density. Thus, $b\mathcal D\sub b\underline{ \mathcal{D}}\sub b\overline{\mathcal {D}}\sub \overline{\mathcal {BD}}$.
    \end{proof}
    \item [c)] $b\mathcal {LD}=b\underline{\mathcal { LD}}=b\overline{\mathcal { LD}}=\overline{\mathcal {BD}}.$
    
    \begin{proof}
    Let $A\in \overline {\mathcal {BD}}$. By $b)$, there is a set $R$ with positive density such that for every finite subset $F$ of $R$ there is $n$ such that $F+n\sub A$. By \eqref{log density vs density}, we have that $R$ has positive logarithmic density. This implies that $A\in {b\mathcal {LD}}.$ Conversely, if $A\in b\overline{\mathcal {LD}}$ then, since $\overline{\mathcal {LD}}\sub \overline{\mathcal D}$, we have that $A\in b\overline{\mathcal D}=\overline{\mathcal {BD}}.$ This proves our claim since   $\mathcal {LD}\sub\underline{\mathcal { LD}}\sub\overline{\mathcal { LD}}$.
    \end{proof}
    
     \item[d)] $b\mathcal Syn=\mathcal {PS}yn$. In particular, $b\mathcal {PS}yn=bb\mathcal Syn=b\mathcal{S}yn=\mathcal {PS}yn$.
     \begin{proof}
     Let $A\in b\mathcal Syn$. Then there is $F\in\mathcal Syn$ such that for every finite set $R\sub F$ there is $n\in\zN$ such that $R+n\sub A$. Let $b$ a bound for the gaps of $F$ and consider for each $k$, a finite set $R_k\sub F$ such that $|R_k|=k$ and $n_k$ such that $R_k+n_k\sub A$. Hence, $\bigcup_{i\in [0,b]} R_k+n_k+i\sub \bigcup_{i\in [0,b]} A+i$ contains an interval of length at least $k$. This implies that $\bigcup_{i\in [0,b]} A+i$ is thick. Equivalently, $A\in \PS$.

     Let $A\in \mathcal {PS}yn$. Then there are $b>0,(n_m)_m$ and $(a_{m,k})_{m,1\leq k\leq m}$ such that
     \begin{enumerate}
         \item [i)] for every $m$ and $1\leq k\leq m$, $n_m+a_{m,1}+\ldots +a_{m,k}\in A$ and
         \item [ii)] for every $m$ and $k\leq m$, $a_{m,k}\in [1,b]$.
     \end{enumerate}
     The sequence $(a_{m,1})_m$ is infinite and belongs to $[1,b]$. Hence, there is $f_1\in [1,b]$ such that $a_{m,1}=f_1$ for infinitely many $m's$. Consider the subsequences of $(n_m)_m$ and $(a_{m,k})_m$, which we keep notating $(n_m)_m$ and $(a_{m,k})_m$ respectively.
     
     Thus,
     \begin{enumerate}
        \item [i)]for every $m$, $n_m+f_1\in A$;
        
         \item [ii)] for every $m$ and $2\leq k\leq m$, $n_m+f_1+ a_{m,2}+\dots+ a_{m,k}\in A$;
         \item [iii)] $f_1\in [1,b]$ and
         \item[iv)] for every $m$ and $2\leq k\leq m$, $a_{m,k}\in [1,b]$.
     \end{enumerate}
     Proceeding by induction we construct for each $m$,  $f_m\in [1,b]$  such that for every $k\le m$,
     $$n_m+f_1+\ldots +f_k\in A.$$
     Let $F=\{\sum_{i=1}^m f_i: m\in \zN\}$. Thus $F$ is syndetic and for every finite set $R\sub F$ there is $n$ such that $R+n\sub A$.
    \end{proof}
     \item [e)] $b\mathcal{IAP}=\kAP$, where $\mathcal{IAP}$ is the family of subsets of the natural numbers containing an infinite arithmetic progression.
     \begin{proof}
          A set $A\in\kAP$ contains arbitrarily long arithmetic progressions of step $k$, for some $k.$ Thus it contains any finite subset of an infinite arithmetic progression of step $k.$ This means that $A\in b\mathcal{IAP}$. Conversely, a set containing (shifts of) any finite subset of an infinite arithmetic progression clearly belongs to $\kAP.$
    \end{proof}
\end{enumerate} 
\end{example}

In \cite[Theorem 8.4]{BonGroLopPer22JFA} it was proved that for  right-shift invariant upper families, an operator is $\F$-hypercyclic if and only if it is hypercyclic and has a residual set of $\F$-recurrent vectors. The next result shows, in particular, that for any block family, the residuality of $\F$-recurrent vectors is not needed. For the family of sets with positive upper Banach density the following proposition also  answers \cite[Question 23]{Pui18}.
\begin{proposition}\label{equivalencia bF hiper}
Let $T$ be an operator and let $\F$ be a family. 
\begin{enumerate}
    \item Let $\F$ be a block family. If $T$ is hypercyclic and almost $\F$-recurrent then $T$ is $\F$-hypercyclic.
    \item Let $\F$ be a left-shift invariant family. If $T$ is $\F$-hypercyclic then $T$ is $\F$-recurrent and hence almost $\F$-recurrent. 
\end{enumerate}
In particular for a left-shift invariant and block family $\F$, we have that an operator is $\F$-hypercyclic if and only if it is almost $\F$-recurrent and hypercyclic.
\end{proposition}
\begin{proof}
$(1)$ Suppose that $T$ is hypercyclic and almost $\F$-recurrent. Let $x$ be a hypercyclic vector and $U$ be a nonempty open set. Let $y\in U$ such that $A=N(y,U)\in \F$.

Let $R$ be a finite subset of $A$. Thus, $W=\bigcap_{r\in R} T^{-r}(U)$ is an open set containing $y$. Since $x$ is a hypercyclic vector there is $n\ge 0$ such that $T^n(x)\in W$ and hence $R+n\sub N(x,U)$. Thus, $N(x,U)\in b\F=\F$ and $x$ is an $\F$-hypercyclic vector.

$(2)$ Let $T$ be $\F$-hypercyclic. Let $x\in X$ be an $\F$-hypercyclic vector and $U$ be a nonempty open set. Let $n\in N(x,U)$ then $T^n(x)$ is an $\F$-recurrent vector in $U$. Indeed, for any open neighbourhood $V$ of $T^n(x)$ we have  $N(T^n(x),V)=N(x,V)-n\in\F$ because $\F$ is left-shift invariant.
\end{proof}
\begin{corollary}\label{F recurrence implica bF hiper}
Let $T$ be an operator and let $\F$ be an hereditary upward family. If $T$ is hypercyclic and almost $\F$-recurrent, then $T$ is $b\F$-hypercyclic.

\end{corollary}
We will see in Theorem \ref{PBF implica PF} that the converse holds for adjoint operators.

 In the next result we focus on the $n$-fold direct sum linear dynamical
system associated to an operator.
\begin{proposition}\label{suma directa P_F}
Let $\F$ be an hereditary upward family and
let $T$ be an operator such that its commutant contains a universal vector, i.e. there is some $v\in X$ such that the set $\{Sv\,:\, ST=TS, \, S:X\to X\, \textrm{continuous mapping}\}$ is dense in $X$.
Then we have the following.
\begin{enumerate}
    \item [a)] If $T$  is almost $\F$-recurrent then $\bigoplus_{i=1}^n T$ is almost $\F$-recurrent.
    \item [b)] If $T$ is  $\F$-recurrent then $\bigoplus_{i=1}^n T$ is $\F$-recurrent.
    \end{enumerate}
    
    In particular, the above statements hold if $T$ is cyclic or if it commutes with a cyclic operator.
\end{proposition}
\begin{remark}
 We note that for  $\F=\F_\infty,$ the family of infinite subsets, this result was proved in \cite[Theorem 9.1]{CosManPar14}   with the additional hypothesis that the set of all vectors $v\in X$ such that  
 $$\overline{\{Sv\,:\, ST=TS, \, S\in\mathcal L(X)\}}=X$$ is residual. See also \cite[Proposition 5.10]{grivaux2022questions}.
 \end{remark}

\begin{proof}
Proof of $a)$. Let $V_1\ldots V_n$ be nonempty open sets and let $v\in X$ be a universal vector for the commutant of $T$. Hence, there are continuous mappings $S_i$, with $1\leq i\leq n$, that commute with $T$ such that $S_i v \in  V_i$, for every $1 \leq  i \leq n$. Thus, $W := \bigcap_{i=1}^n S_i^{-1} V_i$ an
open set containing $v$. Let $x\in W$ be such that $N_T(x,W)\in \F$. If $k\in N_T(x,W)$ then $T^kS_i(x)=S_iT^k(x)\in S_i(W)\sub V_i$ for every $1\leq i\leq n$.
Therefore, $N_T(x,W)\sub \bigcap_{i=1}^n N_T(S_i(x),V_i)$, which implies that 
$N_{\bigoplus_{i=1}^n T}\left(\oplus_{i=1}^n S_i(x),\oplus_{i=1}^nV_i\right)\in \F$.

Proof of $b)$. Let $V_1\ldots V_n$ be nonempty open sets and let $v\in X$ be an universal vector for the commutant. Hence, there are continuous mappings $S_i$ that commute with $T$ and a nonempty open set $W$ containing $v$ such that $S_i(W)\sub V_i$ for every $1\leq i\leq n$. If $x\in W$, then $\oplus_{i=1}^n S_i(x)\in \bigoplus_{i=1}^n V_i$. 

We show that if $x$ is an $\F$-recurrent vector for $T$, then $\oplus_{
i=1}^n S_i(x)$ is
an $F$-recurrent vector for $\oplus_{
i=1}^n T$. Let $U_i$ be open neighbourhoods of $ S_i(x),$ $1\leq i \leq n$. We must show that $\bigcap_{i=1}^n N_T(S_i(x),U_i)\in \F$.
We may assume that $U_i\sub V_i$ for every $1\leq i \leq n.$ Then since $S_i(x)\in S_i(W)\cap U_i,$ there are non-empty open sets $W_i$ such that $x\in W_i\sub W$ and  $S_i(W_i)\sub  S_i(W)\cap U_i$ for every $1\leq i \leq n$. 
Let $\tilde W:=\bigcap_{i=1}^n W_i$. Then 
$S_i(\tilde W)\sub U_i$ for every $1\leq i \leq n$. Since each $S_i$ with $1 \leq i \leq n$ commutes with T, we have that  $N_T(x,\tilde W)\sub \bigcap_{i=1}^n N_T(S_i(x),U_i)$. Finally, noting that $N_T(x,\tilde W)\in \F$, we conclude that 
$\bigcap_{i=1}^n N_T(S_i(x),U_i)\in \F$.

For the last assertion let $R$ be a cyclic  operator with cyclic vector $v$ (i.e. the set $\{p(R)v:\, p\textrm{ polynomial}\}$ is dense in $X$). 
If  $R$  commutes with $T$, then $p(R)T=Tp(R)$ for every polynomial $p$. Then, we have $\{Sv\,:\, ST=TS, \, S:X\to X\, \textrm{continuous mapping}\}\supseteq\{p(R)v\,:\, p\textrm{ polynomial}\}$ which   is dense in $X$ 
\end{proof}
Recall that an operator is said to be weakly mixing provided that $T\oplus T$ is hypercyclic. Equivalently, $\oplus_{i=1}^n T$ is hypercyclic for every $n$ \cite[Theorem 1.51]{GroPer11}.
\begin{corollary}\label{cor: prod F-hiper}
    Let $\F$ be a block family and let $T$ be a weakly mixing and almost $\F$-recurrent operator. Then  $\bigoplus_{i=1}^n T$ is $\F$-hypercyclic for every $n\in\mathbb N$.
\end{corollary}

\begin{proof}
If $T$ is weakly mixing, then it is hypercyclic.
 In particular, $T$ is cyclic. If \emph{T} is almost \emph{F} -recurrent, then by Proposition
 \ref{suma directa P_F} $a)$ we have that $\bigoplus_{i=1}^n T$ is almost $\F$-recurrent. Since $\bigoplus_{i=1}^n T$ is hypercyclic and $\F$ is a block family, it follows by Proposition \ref{equivalencia bF hiper} that  $\bigoplus_{i=1}^n T$ is $\F$-hypercyclic. 
\end{proof}
\begin{remark}
It was shown in \cite[Corollary 2.8]{ErnEssMen21} that an operator $T$ is reiteratively hypercyclic then so is $\bigoplus_{i=1}^n T$  for every $n\in\mathbb N.$
The conclusion of the above corollary in the case $\F=\overline{\mathcal{BD}}$ provides an alternative proof of this fact. Indeed, by Example \ref{ejemplos block families}, $\overline{\mathcal{BD}}$ is a block family. Moreover if $T$ is reiteratively hypercyclic then, by
Proposition \ref{equivalencia bF hiper},  $T$ is 
almost $\overline{\mathcal{BD}}$-recurrent and, by \cite[Proposition 25]{Bes16}, $T$ is weakly mixing.
 \end{remark}
 
The following is the main result of this section. It establishes the converse of Corollary \ref{F recurrence implica bF hiper} for adjoint operators.
\begin{theorem}\label{PBF implica PF}
Let $T$ be an adjoint operator defined on a dual Banach space and let $\F$ be a left-shift invariant  hereditary upward family. If $T$ is almost $b\F$-recurrent then $T$ is almost $\F$-recurrent.
\end{theorem}
\begin{proof}
Let \emph{X} be a dual Banach space and let
\emph{U} be a nonempty open subset of \emph{X}. Let $V\sub \overline V\sub U$ be an auxiliary ball, so that by Alaoglu's Theorem  $\overline V$ is weakly$^*$-compact and convex \cite[Theorem 3.1, Chapter V]{Conway90}. Hence, there is $x\in V$ such that $N(x,V)\in b\F$. By the definition of block family, there is $R\in \F$ such that for every finite subset $F$ of $R$ there is $n\in \mathbb N_0$ such that $F+n\sub N(x,V)$. Let $r_0=\min R$.

Let $R_n=R\cap [0,n]$ and $a_n$ be such that $T^{a_n+r}(x)\in V$ for every $r\in R_n$ and $n\geq 1$. Since $\overline V$ is weakly$^*$-compact  there is a vector $y\in \overline V$ which is a weak$^*$-limit point  of the set $\{T^{a_{n}+r_0}(x)\}$. We claim that $R-r_0\sub N(y,U)$. Indeed, let $r\in R$ and note that $r\in R_n$ (and hence $T^{a_{n}+r}(x)\in V$) for every $n\ge r$. Since $y$ is a weak$^*$-limit point  of the set $\{T^{a_{n}+r_0}(x):\, n\ge r\}$ and $T$ is weak$^*$-weak$^*$-continuous, we have that
 $T^{r-r_0}(y)$ is a weak$^*$-limit point  of the set $T^{r-r_0}\left(\{T^{a_{n}+r_0}(x):\, n\ge r\}\right)=\{T^{a_{n}+r}(x):\, n\ge r\}\subset V$. This implies that $T^{r-r_0}(y)\in \overline V\subset U$ and thus $R-r_0\sub N(y,U)$.
 By left-shift invariance we obtain that $N(y,U)\in \F.$
\end{proof}
The above result is no longer true for non-adjoint operators (see Corollary \ref{coro: existe Reiter hyp no upper recurrent}).

Theorem \ref{PBF implica PF}, together with Corollary \ref{F recurrence implica bF hiper}, gives us the following.
\begin{corollary}\label{P_F sii bF hiper}
Let $T$ be an adjoint operator defined on a separable dual Banach space and let $\F$ be a left-shift invariant hereditary upward family. Then $T$ is hypercyclic and almost $\F$-recurrent  if and only if $T$ is $b\F$-hypercyclic.
\end{corollary}
Recall that a vector $x$ is said to be periodic for an operator $T$ provided that there is $n\in \mathbb N$ such that $T^n(x)=x$.
\begin{corollary}\label{coro to PBF implica PF}
Let $T$ be an adjoint operator defined on a  dual Banach space $X$. Then,
\begin{enumerate}
    \item $T$ is almost ${\mathcal D} $-recurrent  if and only if it is almost $\overline{\mathcal {BD}}$-recurrent;
    \item $T$ is almost ${\mathcal Syn}$-recurrent if and only if it is almost $\mathcal {PS}yn$-recurrent;
    \item $T$ has a dense set of periodic vectors  if and only if it is almost ${\kAP}$-recurrent. 
\end{enumerate}
\end{corollary}
\begin{proof}
We only have to prove (3), since (1) and (2) follow by  Theorem \ref{PBF implica PF} and Example \ref{ejemplos block families} $b)$ and $d)$, respectively.

By  Theorem \ref{PBF implica PF} and Example \ref{ejemplos block families} $e)$ \emph{T} is almost ${\mathcal I\AP}$ -recurrent
if and only if it is $\mathcal{AP}_b$-recurrent. Then, it suffices to
show that if $T$ is almost  ${\mathcal I\AP}$-recurrent then $T$ has a dense set of periodic vectors. Given a nonempty open set $U$ there exists a nonempty open, bounded and convex set $V$ with $\overline V$ weak$^*$ compact satisfying $V\sub \overline V\sub U$ \cite[Theorem 3.1, chapter V]{Conway90}. Since $T$ is almost ${\mathcal I\AP}$-recurrent there are $x\in V$ and $k,n\in\mathbb N$ such that $T^{n+jk}(x)\sub V$ for every $j\ge 0$. 
In other words, $\{ (T^k)^jT^nx : j \ge 0\} := Orb_{T^k} (T^nx) \subset V $. Let $Y :=\overline{co}(Orb_{T^k}(T^nx))\sub \overline V\sub U$ (here, $\overline{co}$ denotes the closed convex hull).
 Note that  $Y$ is weakly$^*$-compact because it is a closed     subset of the weakly$^*$-compact $\overline V$. Thus, since   
$T^kOrb_{T^k}(T^nx)\sub Orb_{T^k}(T^nx)$, we have that $Y$ is a $T^k$-invariant, convex and weakly$^*$-compact set. An application of the  Schauder-Tychonoff Theorem (see \cite{Tyc35}) implies that there is a fixed point $y$ of $T^k$ in $Y$, i.e. $T^ky=y\in U.$ 
\end{proof}

The following result is similar to  Theorem \ref{PBF implica PF}.  We prove it here for the sake of  completeness. Note that since the family $\underline{\mathcal{D}}$ is left-shift invariant (see e.g. \cite{BonGro18}) Theorem \ref{reiterative recurrent implies freq T** BD} follows as an immediate corollary.  We will denote by $j_X$ to the canonical inclusion $X\hookrightarrow X^{**}.$
\begin{theorem}\label{reiterative recurrent implies freq T**}
Let $T$ be a linear operator on a Banach space $X$ and let $\F$ be a left-shift invariant and hereditary upward family. If $T$ is almost  ${b\F}$-recurrent then for each nonempty open set $U\sub X^{**}$ such that $U\cap j_X(X)\ne\emptyset$, there is $z\in U$ such that 
 $N_{T^{**}}(z,U)\in \F$. 
\end{theorem}
\begin{proof}
Let $U\sub X^{**}$ be a nonempty open set such that $U\cap j_X(X)\neq \emptyset.$ Consider $V$ a nonempty open ball of $X$ such that $j_X(V)\sub U$ and whose weak$^*$ closure, $\overline{j_X(V)}^{w*}\sub U$. 
Note that by Goldstine's Theorem (see e.g. \cite{Conway90}), $\overline{j_X(V)}^{w*}$ is a closed ball of $X^{**}$. Therefore, $\overline{j_X(V)}^{w*}$ is a weakly$^*$-compact set \cite{Conway90}.

Let $x\in V$ be such that $N_T(x,V)\in b\F$ and $R\in \F$ be such that for every finite subset $F$ of $R$ there is $n\in\mathbb N_0$ such that $F+n\subset N_T(x,V)$.   Let $R_n=R\cap [0,n]$ and $a_n$ such that $R_n+a_n\subset N_T(x,V)$ for every $n\ge1$. Let $r_0=\min R$. 

 Since $\overline{j_X(V)}^{w*}\sub X^{**}$ is weakly$^*$-compact  there is a vector $z\in \overline{j_X(V)}^{w*}$ which is a weak$^*$-limit point of the set $j_X(\{T^{a_{n}+r_0}(x):n\mathbb N\})$. We claim that $R-r_0\sub N_{T^{**}}(z,U)$. Indeed, let $r\in R$ and note that $r\in R_n$ (and hence $T^{a_{n}+r}(x)\in V$) for every $n\ge r$. Since $z$ is a weak$^*$-limit point  of the set $j_X(\{T^{a_{n}+r_0}(x):\, n\ge r\})$ and $T^{**}$ is weak$^*$-weak$^*$-continuous, we have that
 $(T^{**})^{r-r_0}(z)$ is a weak$^*$-limit point  of the set $(T^{**})^{r-r_0}\left(j_X(\{T^{a_{n}+r_0}(x):\, n\ge r\})\right)=j_X(\{T^{a_{n}+r}(x):\, n\ge r\})\subset j_X(V)$. This implies that $(T^{**})^{r-r_0}(z)\in \overline{j_X(V)}^{w*}\subset U$ and thus $R-r_0\sub N_{T^{**}}(z,U)$.
 By left-shift invariance we obtain that $N_{T^{**}}(z,U)\in \F.$
\end{proof}

Recall that an hereditary upward family $\mathcal F$ is said to have the CuSP property provided that for any finite covering $I_1\ldots I_q$ of $\zN$ and any numbers $n_1,\ldots, n_q$ in $\mathbb N_0$, then
$$A\in \mathcal F \Rightarrow
\Big(\bigcup_{j=1}^q n_j+ \left(A\cap I_j\right)\Big)\in \mathcal F.$$

Families with the CuSP property were introduced in \cite[Section 4]{BonGroLopPer22JFA}, but the concept has its origins in \cite[Section 6.3.3]{BayMat09}. 
It is known that the families $\underline{\mathcal{D}}$, $\overline{\mathcal{D}}$, $\overline{\mathcal{BD}}$, $\mathcal Syn$ and $\mathcal F_\infty$  all have the CuSP property (see \cite[Lemma 4.1]{BonGroLopPer22JFA}), while the families $\AP_b$ and $\mathcal{IAP}$ don't (see \cite[Remark 3.23]{CardeMur22}). Note that the following result has  Proposition \ref{lambda T reiterative} as a  particular case.
\begin{theorem}\label{lambda T CuSP}
Let $\F$ be an hereditary upward family with the CuSP property. If $T$ is almost $\F$-recurrent then, for any $\lambda\in \mathbb K$ with $|\lambda|=1$, $\lambda T$ is almost $\F$-recurrent.
\end{theorem}
\begin{proof}
Let $\lambda\in \mathbb K$ with $|\lambda|=1$.
 Given a recurrent vector $y$ and an open set $U$ containing $y$, we define
$$\Lambda(y,U)=\{\mu\in\mathbb T\,:\, \exists\ n\in\mathbb N \text { such that } T^n(y)\in U \text{ and } \mu=\lambda^n\}.$$

We also consider $\Lambda(y)$ the intersection of the closures of $\Lambda(y,U)$, i.e., $$\Lambda(y)=\bigcap_{ 
 U \text{ open, } y\in U} \overline {\Lambda (y,U)}.$$

Note that $\Lambda(y)\neq \emptyset$. Indeed, since the $\overline {\Lambda (y,U)}$ are all compact, it suffices to show that the collection  $\{\overline{\Lambda(y,U)}:U \text{ is an open set containing } y\}$ has the finite intersection property.
Let $U_1,\ldots U_N$ be open sets containing $y$. Let $U'=\bigcap_{j=1}^N U_j$, which is nonempty because $y\in U'$. Since $y$ is recurrent, it follows that there is $n\in \zN$ such that $T^{n}(y)\in U'\subseteq U_j$ for every $1\leq j\leq N$. This shows that $\lambda^n\in \bigcap_{j=1}^N \Lambda (y,U_j)\subseteq \bigcap_{j=1}^N \overline{\Lambda (y,U_j)}$.

Proceeding as in \cite[Theorem 4.4]{BonGroLopPer22JFA}, it follows that 
 $\Lambda(y)$ is a multiplicative semigroup of the unit circle.

Indeed, let $\mu_1,\mu_2\in \Lambda(y)$. Let $U_1$ be an open set containing $y$ and let $\varepsilon>0$. Since $\mu_1\in\overline{\Lambda(y,U_1)},$
there is $n_1\in\zN$ such that $|\lambda^{n_1}-\mu_1|<\varepsilon$ and such that $T^{n_{1}}(y)\in U_1$. Let $U_2$ be an open set containing $y$ such that $T^{n_1}(U_2)\sub U_1$. Since $\mu_2\in \Lambda(y)\sub \overline {\Lambda(y,U_2)}$, there is $n_{2}\in\zN$ such that $T^{n_{2}}(y)\in  U_2$ and $|\lambda^{n_2}-\mu_2|<\varepsilon$. Altogether we find that $T^{n_{1}+n_{2}}(y)\in T^{n_{1}}(U_2)\subset U_1$. Hence, $\lambda^{n_{1}}\lambda^{n_2}\in \Lambda (y,U_1)$. Finally, 
$$|\mu_1\mu_2-\lambda^{n_1}\lambda^{n_2}|\le |\mu_1(\mu_2-\lambda^{n_2})|+|(\mu_1-\lambda^{n_1})\lambda^{n_2}|<2\varepsilon.$$
This implies that $\mu_1\mu_2 \in \overline{\Lambda(y,U_1)}$ for every open set $U_1$ containing $y$. Therefore $\mu_1\mu_2\in\Lambda(y).$

It is known that any closed multiplicative subsemigroup of $\mathbb T$ is either $\zT$ or $G_n$ for some $n\in\mathbb N$ (the group of $n$-th roots of the unity), see for example \cite[p.170]{GroPer11}. Thus, for each $y,$ we have $\Lambda(y)=\zT$ or $\Lambda(y)=G_n$ for some $n\in\mathbb N$.
 Note that for real Banach spaces we are only concerned with semigroups of $\mathbb T\cap \mathbb R$, thus the only possibilities are $\Lambda(y)=G_1$ or $\Lambda(y)=G_2.$
 
Recall that since $T$ is almost $\F$-recurrent, it is recurrent \cite[Proposition 2.1]{CosManPar14}.
Let  $U$ be a nonempty open set and consider an open ball $U'$ and $\varepsilon>0$ such that $B(1,\varepsilon)U'\sub U,$ where $B(1,\varepsilon)$ denotes the open disc of center 1 and radius $\varepsilon.$ We will show that there is $x\in U'$ such that $N_{\lambda T}(x,U)\in \mathcal F.$

We will divide the proof in three cases: 

\textbf{Case 1: There is a recurrent vector $y\in U'$ such that $\Lambda(y)=\zT.$}

This case can be proved as in \cite[Theorem 4.4]{BonGroLopPer22JFA}.
Let $y\in U'$ be a recurrent vector such that $\Lambda(y)=\zT.$ In particular, we have that $\Lambda (y,V)$ is dense in $\zT$, for every open set $V$ containing $y$.

Thus, for every $\mu\in \zT$  there is  $n\in \zN$ such that $T^n(y)\in U'$ and such that $|\lambda^n-\mu|<\varepsilon$. By the compactness of $\mathbb T$,  there are $N\in \zN$ and natural numbers $n_j$,  with $1\leq j\leq N$, such that $T^{n_j}(y)\in U' $ for every $1\leq j\leq N$ and such that 
\begin{equation}\label{compactness}
    \zT\sub \bigcup_{j=1}^N B(\lambda^{n_j},\varepsilon).
\end{equation}
Let $W$ be an open set containing $y$ such that $T^{n_j}(W)\sub U'$ for every $1\leq j\leq N$. Let $x\in W$ be such that $A:=N_T(x,W)\in\F$. Note that
by \eqref{compactness}, the sets
$I_j=\{n: \lambda^{n+n_j}\in B(1,\varepsilon)\}$, with $1\leq j\leq N$, form a cover of $\zN$. Indeed, if $n\in \mathbb N$ there is $1\leq j_0\leq N$ such that $\lambda^{-n}\in B(\lambda^{n_{j_0}},\varepsilon).$ Thus, $|\lambda^{n+n_{j_0}}-1|\leq |\lambda^{n}||\lambda^{n_{j_0}}-\lambda^{-n}|<\varepsilon,$ which implies that $n\in I_{j_0}.$

If $n\in A\cap I_j$ we have that
\begin{align*}
    \lambda^{n+n_j} T^{n+n_j}(x)\in B(1,\varepsilon) T^{n_j}(W)\sub B(1,\varepsilon) U' \sub U.
\end{align*}
We have proved that $\bigcup_{j=1}^N n_j+(I_j\cap N_T(x,W))\subset N_{\lambda T}(x,U)$.
It follows by the CuSP property that $N_{\lambda T}(x,U)\in\F$. This proves the first case.

The case when $\Lambda(y)\ne \mathbb T$ is rather more involved than the $\F$-recurrence analogue in \cite[Theorem 4.4]{BonGroLopPer22JFA}, and we must consider two different situations. For a finite group $G$, let $\textrm{\bf{ord}}(G)$ denote its cardinality.

\textbf{Case 2: $\sup\{ \textrm{ord}(\Lambda(y)):y\in U'\}=\infty$ and $\Lambda(y)\neq \zT$ for every recurrent vector $y\in U'$.}

Let $n_0$ be such that
\begin{equation}\label{cubrimiento}
    \zT\sub \bigcup_{j=0}^{n-1} B(e^{2\pi i \frac{j}{n}},\frac{\varepsilon}{2})
\end{equation}
for every $n\geq n_0$.

By assumption, there are a recurrent vector $y\in U'$ and $N\geq n_0$ such that $\Lambda(y)=G_N$. Since each $e^{2\pi i\frac{j}{N}}\in\Lambda(y)$, there is for each $0\leq j\leq N-1$, an $n_j\in \mathbb N$ such that $T^{n_j}(y)\in U'$ and such that $|\lambda^{n_j}-e^{2\pi i\frac{j}{N}}|<\frac{\varepsilon}{2}.$

It follows by \eqref{cubrimiento} that for every $\mu\in \zT$, there is $0\leq j\leq N-1$ such that $|\mu-\lambda^{n_j}|<\varepsilon$ and such that $T^{n_j}(y)\in U'.$ So, we obtain \eqref{compactness} and conclude as we did in
the first case.

\textbf{ Case 3: }$\sup \{\textbf{\textrm{ord}}(\Lambda(y)): y\in U',\, y\textrm{ recurrent}\}<\infty.$

We first show that there exists an open set $U''$  such that $U''\subseteq U'$ with the property that there are $G_{r_1},\ldots, G_{r_N}$ satisfying that for all $n\in\zN$: 
\begin{equation}
\label{eq:U''}
\text{the set }\{y\in U'':y\textrm{ recurrent and }\Lambda(y)=G_{n}\}\text{ is} 
\left\{\begin{array}{ll}
\text{dense in }U'',    & \text{  if }n\in\{r_1,\dots,r_N\},  \\
  \text{empty},   &  \text{  if }n\notin\{r_1,\dots,r_N\}.
\end{array}
\right.
\end{equation}

Indeed, since $\sup \{\textrm{\bf{ord}}(\Lambda(y)):y\in U',\, y\textrm{ recurrent}\}<\infty$ and $\Lambda(y)\ne\emptyset$ for each recurrent vector $y$, there are  $G_{s_1},\ldots G_{s_K}$ such that for every recurrent vector $y\in U'$ there is $1\leq k\leq K$ such that $\Lambda(y)= G_{s_k}$. For each $1\leq k\leq K$, consider $D_{k,U'}=\{y\in U':y\textrm{ recurrent and }\Lambda(y)=G_{s_k}\}$.
Since $\bigcup_{1\leq k\leq K} D_{k,U'}=\{ \text {recurrent vectors of $T$ in } U'\}$ and the closure of the set of recurrent vectors for $T$ contains
 $U'$, there must be some $1\leq k_0\leq K$ such that $\overline{D_{{k_0},U'}}$ has nonempty interior. 
To simplify notation, we suppose that $k_0=1$. Let $U_1$ be the interior of $\overline{D_{{1},U'}}$. Note that $U_1\sub U'$.

If $D_{2,U_1}:=\{y\in U_1:y\textrm{ recurrent and } \Lambda(y)=G_{s_2}\}$ is dense in $U_1$, we set $U_2=U_1$. Otherwise, consider a nonempty open set $U_2\subset U_1$ such that $U_2\cap \overline{D_{2,U_1}}=\emptyset$.
Inductively, if $D_{j,U_{j-1}}:=\{y\in U_{j-1}: y\textrm{ recurrent and }\Lambda(y)=G_{s_j}\}$ is dense in $U_{j-1}$, we set $U_j=U_{j-1}$. Otherwise, consider a nonempty open set $U_j\subset U_{j-1}$ such that $U_j\cap \overline{D_{j,U_{j-1}}}=\emptyset$.

We have constructed nonempty 
open sets $U_1\supseteq U_2\supseteq \dots\supseteq U_K$ such that  for each $1\le j\le K,$ the sets $D_{j,U_K}:=\{y\in U_K: y\textrm{ recurrent and }\Lambda(y)=G_{s_j}\}$ are either dense in $U_K$ or empty. Then if  $G_{r_1},\dots,G_{r_N}\in\{G_{s_1},\dots,G_{s_K}\}$ are the groups $G_{r_j}$ such that $\{y\in U_K: y\textrm{ recurrent and }\Lambda(y)=G_{r_j}\}$ is  dense in $U_K$,  the set $U''=U_K$ satisfies \eqref{eq:U''}.

Given a nonempty open set $W$, consider
$$\beta(W):=\{\mu\in\mathbb T: \exists n\in\mathbb N \text{ and a nonempty open set } V\sub W \text{ such that } T^n(V)\subseteq W, \mu=\lambda^n\}.$$

Note that 
\begin{equation}\label{eq: beta es la union de los lambda}
  \beta(W)=\bigcup_{ y\in Rec(W)} \Lambda(y,W),
\end{equation}
where $Rec(W)$ denotes the set of recurrent vectors in $W$.
 Indeed, if $y$ is a recurrent vector in $W$ and $T^n(y)\in W$, then there is an open set $V$ around $y$ such that $T^n(V)\subseteq W$. This implies that $\bigcup_{ y\in Rec(W)} \Lambda(y,W) \subseteq \beta(W)$.

Conversely, if $\mu\in \beta(W)$, then there is $n\in\mathbb N $ and a nonempty open set  $V\subseteq W$ such that $T^n(V)\subseteq W$ and such that $\mu=\lambda^n$. Let $y\in V$, be a recurrent vector. Hence, $T^n(y)\in W$ and $\mu=\lambda^n$. We conclude that $\mu\in \Lambda(y,W).$ Hence \eqref{eq: beta es la union de los lambda} follows.

Let $U'',G_{r_1},\dots,G_{r_N},$ satisfying \eqref{eq:U''} and define
$$\beta=\bigcap_{W\subseteq U''}\overline{\beta(W)}.$$ 
\textbf{Claim} $\beta=\bigcup_{k=1}^N G_{r_k}$. 

Indeed, let $\mu \in \beta$. Then, there are a sequence natural numbers $(n_j)_{j\ge1}$ and a decreasing chain of nonempty open sets $V_{n_j}\subseteq \overline {V_{n_j}}\subseteq V_{n_{j-1}}$ (which can be taken to be open balls whose diameter goes to 0) such that $T^{n_j}(V_{n_j})\subseteq V_{n_{j-1}} $, $\lambda^{n_j}\in\beta(V_{n_{j-1}})$ and $\lambda^{n_j}\to \mu$. Since  $(\overline{V_{n_j}})_{j\ge1}$ have diameter decreasing to 0, it follows, by the well known Cantor’s intersection theorem for complete metric spaces, that there is $y\in U''$ such that $\cap_{j\ge1} \overline{V_{n_j}}=\{y\}$. Since $(V_{n_{j}})_{j\ge1}$ is a neighbourhood basis of $y$ and $T^{n_j}(y)\in V_{n_{j-1}}$ for every $j\in \mathbb N$, we have that $y$ is a  recurrent vector. This implies that $\mu \in \overline{\bigcap_{V \text{ open }, V\ni y} \Lambda (y,V)}\subseteq\Lambda(y)=G_{r_k}$ for some $1\leq k\leq N$.

Conversely, let $1\leq k\leq N$ and let $W\subseteq U''$ be a nonempty open set. By \eqref{eq:U''}, there is a recurrent vector $y\in W$ such that $\Lambda(y)=G_{r_k}$. Therefore, $G_{r_k}\subseteq \overline{\Lambda(y,W)}\subseteq \overline{\beta(W)}.$ This implies that $G_{r_k}\subseteq \beta.$

We will show now that there is $W'\subseteq U''$ such that 
\begin{equation}
  \label{cubrimiento beta}  
\beta(W')\subseteq \bigcup_{1\leq k\leq N, 0\leq j\leq r_{k}-1} B(e^{2\pi i\frac{j}{r_k}},\frac{\varepsilon}{2})= 
\bigcup_{s\in G_{r_k}, 1\leq k\leq N} B(s,\frac{\varepsilon}{2}).
\end{equation}
Suppose otherwise, then 
for each nonempty open set $W\subseteq U''$, there is some $n_W$ such that $\lambda^{n_W}\in\beta(W)\setminus \bigcup_{s\in G_{r_k}, 1\leq k\leq N} B(s,\frac{\varepsilon}{2}),$ which implies that there is some open set $V_W\sub W$ such that $T^{n_W}(V_W)\sub W$ and $|\lambda^{n_W}-s|\ge\frac{\varepsilon}{2}$ for every ${s\in G_{r_k}, 1\leq k\leq N}$. Hence, we can construct inductively a sequence of
natural numbers $(n_j)_{j\ge 1}$ and a decreasing chain of nonempty open sets $(V_{n_j})_{j\ge 1}\subset U''$ whose diameter tend to 0, with $\overline{V_{n_j}}\subseteq V_{n_{j-1}}$ such that $T^{n_j}(V_{n_j})\subseteq V_{n_{j-1}}$, and such that  $|\lambda^{n_j}-s|\ge\frac{\varepsilon}{2}$ for every $s\in \bigcup_{k=1}^N G_{r_k}$. Moreover, passing to a subsequence of $(\lambda^{n_j})_{j\ge1}$, we may assume that there is $\mu\in\zT$ such that $\lambda^{n_{j}}\to \mu$. 

Consider $\{y\}=\bigcap \overline{V_{n_j}}$. Since for every $j\ge 1$, $T^{n_j}y\in T^{n_j}V_{n_j}\sub V_{n_{j-1}}\sub\dots \sub V_{n_{1}}$, we have that $y$ is a recurrent vector and that $\lambda^{n_{j}}\in \Lambda(y,V_{n_{l}})$ for every $l$ such that $1\le l<j$. Thus $\mu\in \overline{\Lambda(y,V_{n_{l}})}$ for every $l\in \zN$ and hence $\mu\in \Lambda(y)$ (because for all open sets $U$ containing $y$ there exists $l\ge 1$ such
that $V_{n_l} \sub U$). 
Since $\Lambda(y)=G_{r_{k_0}}$ for some $1\leq k_0\leq N$, we have that $\mu\in \bigcup_{k=1}^N G_{r_k}$. This  is a contradiction because we assumed that $|\lambda^{n_{j}}-s|\ge\frac{\varepsilon}{2}$ for every $s\in \bigcup_{k=1}^N G_{r_k}.$

So, let $W'\sub U''$ be an open set satisfying \eqref{cubrimiento beta}.

Let $y_1\in W'$ be a recurrent vector such that $\Lambda(y_1)=G_{r_1}$. Therefore, by the definition of $\Lambda(y_1)$,  there is  for each $0\leq j\leq r_1-1$ a natural number $n_{1,j}$ such that $T^{n_{1,j}}(y_1)\in W'$ and $|\lambda^{n_{1,j}}-e^{-2\pi i \frac{j}{r_1}}|<\frac{\varepsilon}{2}.$
For each $0\leq j\leq r_1-1$, let $W_{1,j}$ be an open set  that contains $y_1$ such that $T^{n_{1,j}}(W_{1,j})\subseteq W'$ and let $W_1:=\cap_{j=0}^{n_1} W_{1,j}\subset W'$, which is nonempty because $y\in W_{1,j}$ for every $0\leq j\leq r_1-1$.

Take now $y_2\in W_1$ be a recurrent vector such that $\Lambda(y_2)=G_{r_2}$. Proceeding as before we obtain an open set $W_2\sub W_1$  and  natural numbers $n_{2,j}$ for $0\leq j\leq r_2-1$, such that $T^{n_{2,j}}(W_{2})\subseteq W_1$ and $|\lambda^{n_{2,j}}-e^{-2\pi i \frac{j}{r_2}}|<\frac{\varepsilon}{2}.$

Repeating the argument for each $G_{r_k}$, $1\le k\le N$, we construct a nonempty open set $W'':=W_N\subset W'$ and natural numbers ${n_{k,j}}$ with $1\leq k\leq N$ and $0\leq j\leq r_k-1$ such that 
\begin{equation}\label{W''}
T^{n_{k,j}}(W'')\subseteq W' \text{ and } |\lambda^{n_{k,j}}-e^{-2\pi i \frac{j}{r_k}}|<\frac{\varepsilon}{2}
\end{equation}
 for every $1\leq k\leq N,$ $ 0\leq j\leq r_k-1.$    

Since $T$ is almost $\F$-recurrent there is some  $x\in W''$ 
such that $A:= N_T(x,W'')\in\mathcal F$. We will prove that $N_{\lambda T}(x,U)\in \F$.

For each $1\leq k\leq N$ and $0\leq j\leq r_k-1$ consider 
$$I_{k,j} = \{n\in\zN : \lambda^{n+n_{k,j}}\in  B(1, \varepsilon)\}.$$

Let $n\in A$. By continuity, there is an open set $V\subseteq W''$ around $x$ such that $T^n(V)\subseteq W''$. Thus, $$\lambda^n\in \beta(W'')\subseteq  \beta(W')\subseteq \bigcup_{1\leq k\leq N, 0\leq j\leq r_{k}-1} B(e^{2\pi i\frac{j}{r_k}},\frac{\varepsilon}{2}),$$
where the last inclusion follows by \eqref{cubrimiento beta}.

Thus, there is $1\leq k_0\leq N$ and $0\leq j_0\leq r_{k_0}-1$ such that
$$|\lambda^{n}-e^{2\pi i \frac{j_0}{r_{k_0}}}|<\frac{\varepsilon}{2}.$$
From this inequality and  \eqref{W''} we     conclude that
\begin{equation}\label{1+epsilon}
|\lambda^{n+n_{k_0,j_0}}-1|=|\lambda^{-n}-\lambda^{n_{k_0,j_0}}|\leq 
    |\lambda^{-n}-e^{-2\pi i \frac{j_0}{r_{k_0}}}|+|e^{-2\pi i \frac{j_0}{r_{k_0}}}-\lambda^{n_{k_0,j_0}}|<\varepsilon.
\end{equation}
This means that $n\in I_{k_0,j_0}.$
Therefore,  $A\subset \bigcup_{1\leq k\leq N,0\leq j\leq r_k-1} I_{k,j}$ and hence the sets $I_{k,j}$ with ${1\leq k\leq N}$, ${0\leq j\leq r_k-1}$  together with $\mathbb N\setminus A$ form a finite covering of $\mathbb N.$

Finally, if $n\in A\cap I_{k,j}$ we have by \eqref{W''} and \eqref{1+epsilon} that
$$    \lambda^{n+n_{k,j}} T^{n+n_{k,j}}(x)\in B(1,\varepsilon) T^{n_{k,j}}(W'')\subseteq B(1,\varepsilon) W' \subseteq U,
$$
i.e. $n+n_{k,j}\in N_{\lambda T}(x,U)$.

We have proved that 
\begin{align*}
  \big((\mathbb N\setminus A)\cap N_T(x,W'')\big)  \cup\Bigg(&{\bigcup_{1\leq k\le N,0\leq j\leq r_k-1}  n_{k,j}+(I_{k,j}\cap N_T(x,W''))}\Bigg)\\
  & =\bigcup_{1\leq k\le  N,0\leq j\leq r_k-1} n_{k,j}+(I_{k,j}\cap A)\\ 
  & \subset\quad N_{\lambda T}(x,U).
\end{align*}
It follows by the CuSP property that $N_{\lambda T}(x,U)\in\mathcal F$.
\end{proof}

Consider $\delta>0$
and define $\overline{\mathcal{BD}}_\delta = \{A \subseteq\mathbb N : \overline{\mathcal{BD}}(A) \ge\delta\}$. It is known that no Banach space supports a $\overline{\mathcal{BD}}_\delta$-hypercyclic operator, as shown in \cite[Section 3]{Bes16}. However, if $T^k$ is
the identity operator with $k\le\delta^{-1}$, then $T$ is
$\overline{\mathcal{BD}}_\delta$-recurrent. Specifically, for any open set
$U$ containing $x$, we have $k\mathbb N_0\subseteq  N_T(x,U)$, implying that $N_T(x,U)\in \overline{\mathcal{BD}}_\delta$. In fact, we
will prove that these operators are the only ones that exhibit
almost $\overline{\mathcal{BD}}_\delta$-recurrence.
 \begin{lemma}\label{lema B_delta}
Let $T$ be a continuous mapping on  a nonempty connected open set  $V$ and let $\delta>0$ be such that for every nonempty open set $U\sub V$ there is $x_U\in U$ such that $N(x_U,U)\in\overline{\mathcal{BD}}_\delta$. Then, there is $j\le \delta^{-1}$ such that $T^j|_V=Id_V,$ where $Id_V$ is the identity operator on $V.$
\end{lemma}
\begin{proof}
Let’s proceed by contradiction. Consider $k>\delta -1$. We will demonstrate that for every
nonempty open set $U \subseteq V$, there exists $j$ with $1\le j \le k$ such that $T^j (U )\cap U \ne\emptyset$. Indeed,
note that $\overline{\mathcal{BD}}_\delta=\bigcup_{r\ge\delta}\bigcap_{N\ge1}\mathcal A_{r,N}$, where $A\in \mathcal A_{r,N}$ if and only if there exists $m \ge 0$ such
that $\frac{\#\{A \cap [m, m + N ]\}}{(N + 1)} > r$. By assumption, there exists $x_U \in U$ such that
$N (x_U , U )\in \overline{\mathcal{BD}}_\delta$. Specifically, $N (x_U , U ) \in A_{r,k}$ for some $r\ge\delta$. Hence, there exists $m \ge 0$
such that $\#\{N (x_U , U ) \cap [m, m + k]\} > (k + 1)r > k\delta > 1$. In other words, there exist $n\ge 0$
and $n + j$ with $1 \le j \le k$ such that $\{n, n + j\} \subset N (x_U , U )$. Therefore, $T^n(x_U )\in U$ and
$T^j (T^n(x_U ))\in U.$

Suppose that there exists $x\in V$ such that $T^j(x)\neq x$ for all $j$ with $1\le j\leq k$. Thus there are $V_j$ open neighbourhoods of $T^jx$ such that $x\notin \overline{V_j}$, for $j=1,\dots,k.$ By continuity, there is some open neighbourhood $V_0\sub V$ of $x$ such that $T^j(V_0)\sub V_j$ and $V_0\cap V_j=\emptyset$ for all $j$ with $1\le j\leq k$. This is a contradiction, because it implies that $T^j(V_0)\cap V_0=\emptyset $ for $j=1,\dots, k.$

This implies that every vector in $V$ is  periodic  of period at most $ k$.
It is easy
to check that the set $\{x \in {V} : {T^jx} = x\}$ of
periodic vectors of period ${j}$ contains all its accumulation
points; hence, it is closed. Therefore, the connectivity of ${V}$
 implies that all vectors in $V$ must have period $j$, for some $j=1,\dots,k,$ i.e. $T^j|_V=Id_V$.  
\end{proof}

The fact that any Banach space is connected combined with Lemma 3.17, yields the following
result.

\begin{corollary}\label{BD>delta}
Let $T$ be an operator acting on a Banach space $X$. If $T$ is almost   $\overline{\mathcal{BD}}_\delta$-recurrent, then  $T^k= Id_X$, for some  $k\le \delta^{-1}$.
\end{corollary}

\section{$\F$-recurrent backward shift operators \label{section shifts}}\label{section backwardshifts}
In this section we characterize almost $\F$-recurrent unilateral and bilateral backward shift operators. 

In the following we will consider  a Banach space  $X$ with basis $\{e_n:\, n\in\mathbb N_0\}$ where the (unweighted) backward shift $B$, defined by $B(e_n)=e_{n-1}$ for $n\geq 1$ and $B(e_0)=0$, acts continuously. We remark that we will mostly work with unweighted shifts for simplicity, but all the results can be translated to the weighted setting through a conjugacy argument. Indeed, given any weighted backward shift with weights $(w_n)_{n\ge 1}$ on $X$ defined by $B_w(e_n)=w_ne_{n-1}$ for $n\geq 1$ and $B_w(e_0)=0$, we can consider the unweighted backward shift $B$ acting on the space $X(v):=\{x=\sum_{n=0}^\infty x_ne_n\,:\,\sum_{n=0}^\infty x_nv_ne_n\in X\},$ where $v_n=(w_1\dots w_n)^{-1}$ for $n\ge 1$ and $v_0=1$. We transfer the topology from $X$ to $X(v)$ so that the the canonical mapping $\phi_v: X(v)\to X$, $(\phi_v(x))_n\mapsto (x_nv_n)_n$ is an isomorphism and it  satisfies that $\phi_v B=B_w\phi_v$.  Then $B:X(v)\to X(v)$ and $B_w:X\to X$ are factors of each other so their dynamic is equivalent (see e.g. \cite[Section 4.1]{GroPer11}). In particular, $B$ is either $\F$-hypercyclic, $\F$-recurrent or almost $\F$-recurrent if and only if $B_w$ has the same property.
\begin{center}
\begin{tikzcd}
    X(v) \arrow{r}{B} \arrow{d}[swap]{\phi_v} & X(v) \arrow{d}{\phi_v} \\
    X \arrow{r}[swap]{B_w} & X
\end{tikzcd}
\end{center}

It was proved in \cite[Theorem 8.5]{BonGroLopPer22JFA} that if there is a dense set of vectors satisfying  $T^n(x)\to 0$, then any $\F$-recurrent operator is also $\F$-hypercyclic whenever $\F$ is a uniformly finitely invariant upper family. The same proof holds for almost $\F$-recurrent operators. Recall that an upper family $\F=\bigcup_{\delta\in D} \A_\delta,$ is said to be uniformly finitely invariant provided that if $A\in\F$ then there is some $\delta\in D$ such that for every cofinite subset $C$ we have that $A\cap C\in \F_\delta$, see \cite[Definition 2.9]{BonGro18}.
\begin{proposition}\label{vectores densos tienden a cero}
 Let $\F=\bigcup_{\delta\in D} \A_\delta$ be a uniformly finitely invariant upper  family and let $T$ be a linear operator such that there is a dense set of vectors with $T^n(x)\to 0$. If $T$ is almost $\F$-recurrent then it is $\F$-hypercyclic.

\end{proposition}
\begin{proof}
Let $U,V$ be nonempty open sets, and take $W_0,W$ nonempty open sets such that $0\in W_0$ and $W+W_0\sub V$. Hence, there is $y\in W$ and $\delta\in D$ such that $N(y,W)\cap C\in \F_\delta$ for every cofinite set $C$.  Take $x\in U-y$ such that $T^n(x)\to 0$. Therefore $x+y\in U$ and for large enough $n$, $T^n(x)\in W_0$. Hence if $n\in N(y,W)$ is large enough, $T^n(x+y)\in W+W_0\sub V$. Therefore, $N(x+y,V)\in \F_\delta$. By Theorem \ref{equivalencias upper} this implies that $T$ is $\F$-hypercyclic.
\end{proof}
We note that Proposition \ref{vectores densos tienden a cero} does not necessarily hold for non-upper families, for instance, it is not true for the family of syndetic sets (see Remark 
\ref{rem: ejemplo non upper}).

\begin{corollary}\label{coro: w* cont con Tnx->0, PbF sii F-hyp}
Let $T$ be an adjoint operator on a separable Banach space with a dense set of vectors such that $T^n(x)\to 0$ and let $\F$ be an upper uniformly finitely invariant family. If $T$ is almost ${b\F}$-recurrent  then it is 
$\F$-hypercyclic. 
\end{corollary}
\begin{proof}
Since by definition upper families are left-shift invariant,  Theorem \ref{PBF implica PF} implies that $T$ almost $\F$-recurrent. Hence, by  Proposition \ref{vectores densos tienden a cero} $T$ is $\F$-hypercyclic.
\end{proof}

\begin{corollary}\label{coro: existe Reiter hyp no upper recurrent}
    There exists a unilateral weighted backward shift  on $c_0$ which is $\mathcal{\overline{BD}}$-hypercyclic (and in particular almost $\mathcal{\overline{BD}}$-recurrent) but not almost $\mathcal{\overline{D}}$-recurrent.
\end{corollary}
\begin{proof}
In \cite{Bes16} an example of a reiteratively hypercyclic weighted backward shift over $c_0$ which is not  upper frequently hypercyclic  was given. Since any unilateral weighted backward shift has a dense set of vectors such that $B_w^{n}(x)\to 0$, Proposition \ref{vectores densos tienden a cero} implies that this operator is not almost ${\overline {\mathcal D}}$-recurrent. 
\end{proof}

In \cite{Bes16} it was proved that for $\ell_p$ spaces, reiteratively hypercyclic weighted backward shifts are indeed frequently hypercyclic. The main Theorem of \cite{charpentier2019chaos} states that, on spaces with boundedly complete unconditional basis,  upper frequently hypercyclic weighted backward shifts are both frequently hypercyclic and chaotic (recall that chaotic operators are those which are hypercyclic and have a dense set of periodic vectors). 
Recall that a basis $\{e_n\}$ is boundedly complete if it satisfies that for each sequence of scalars $(a_n)_n$ such that $\Big(\sum_{0\le k\le n} a_ke_k \Big)_{n\ge 0}$ is bounded,  the series $\sum_{k=0}^\infty a_ke_k $ is convergent. Recall also that a basis $\{e_n\}$ is unconditional with unconditional constant $C$ if for all $n\in\mathbb N,$ $\Big\|\sum_{0\le k\le n} a_ke_k \Big\|\le C\Big\|\sum_{0\le k\le n} b_ke_k \Big\|$ whenever $a_0,\dots,a_n,b_0,\dots,b_n$ are scalars such that $|a_k|\le|b_k|$ for $0\le k\le n$.  
\begin{theorem}
Let $X$ be a Banach space with  boundedly complete basis $\{e_n:\, n\in\mathbb N_0\}$. Let $B_w$ be a weighted shift on $X$. Then the following are equivalent.
\begin{enumerate}
    \item [i)] $B_w$ is almost  reiteratively recurrent;
    \item [ii)] $B_w$ is reiteratively hypercyclic;
    \item [iii)] $B_w$ is upper frequently
    hypercyclic;
\item [iv)] $B_w$ is $\overline{\mathcal{LD}}$-   hypercyclic.    
\end{enumerate}
If moreover, the basis is unconditional, then i)-iv) are equivalent to 
\begin{enumerate}
    \item [v)] $B_w$ is frequently hypercyclic and
    \item [vi)] $B_w$ is chaotic.
\end{enumerate}
\end{theorem}
\begin{proof}
As usual, it is sufficient to prove the result in the unweighted case.
Since $b\overline{\mathcal{LD}}=b\overline{\mathcal{D}}=\overline{\mathcal{BD}}$, by Corollary \ref{coro: w* cont con Tnx->0, PbF sii F-hyp}, it suffices to show that $B$ is an adjoint operator.
It is well known that $\{e_n\}$ is boundedly complete if and only if $X$ is isomorphic to $H^*$, where $H=\overline{ span\{e_n^*\}}\sub X^*$ and $\{e_n^*\}$ denotes the dual basis. The isomorphism is given by the mapping $j:X\to H^*$, defined as $j(x)(h)=h(x)$ for $x\in X$ and $h\in H$, see e.g. \cite[Theorem 3.2.10]{AlbKal06}.

Consider $S:=B^*:X^*\to X^*$, or equivalently $B^*:H^{**}\to H^{**}$.  Now we notice that $S(e_n^*)=e_{n+1}^*$. This implies that $S(H)=S(\overline{ span\{e_n^*\}})\sub \overline{ span\{e_n^*\}}=H$, i.e.
$S|_H:H\to H$ is well defined and hence $(S|_H)^*=B$.

The last assertion now follows from  \cite[Theorem 2.1]{charpentier2019chaos}. 
\end{proof}

Recall that an hereditary upward family $\F$ is said to be finitely invariant if $A\cap[m,\infty)\in \F$ for every $A\in \F$ and every $m\in \zN$.
We now
turn our attention to the study of almost
$\F $-recurrent unweighted backward shifts, where $\F $ is a
finitely invariant hereditary upward family.
\begin{theorem}\label{ carac PF uni}
Let $X$ be Banach space with  basis $\{e_n:\, n\in\mathbb N_0\}$ and suppose that the backward shift is well defined. Let $\F$ be a finitely invariant hereditary upward family.  Consider the following statements.
\begin{enumerate}
    \item [i)]$B$ is almost $\F$-recurrent;
    \item [ii)]for every $\varepsilon>0$ and every $p\in\mathbb N$ there is a set $A_{p,\varepsilon}\in \F$ such that
\begin{enumerate}
\item $\sum_{n\in A_{p,\varepsilon}} e_{n+p}$ converges and
\item for any $m\in A_{p,\varepsilon}$, $$\Big\|\sum_{n>m,n\in A_{p,\varepsilon}} e_{n-m+p}\Big\|<\varepsilon.$$
\end{enumerate}
\end{enumerate}
   Then $ii)\Rightarrow i)$. If the basis is unconditional $ii)\Leftrightarrow i)$.
\end{theorem}
Condition $ii)$ above is also equivalent to $B$ being almost $\F$-recurrent and hypercyclic. Indeed,  almost $\F$-recurrence implies recurrence (see \cite[Theorem 2.1]{CosManPar14}), and recurrent shifts are hypercyclic (see e.g. \cite{bonilla2020zero,BonGroLopPer22JFA,ChaSec12} ).

Bonilla and Grosse-Erdmann in \cite[Theorem 5.1]{BonGro18} showed that if $\F$ is a  uniformly finitely invariant upper family then    $\F$-hypercyclic backward shifts are exactly those that satisfy condition $ii)$ of Theorem \ref{ carac PF uni}. 
We have thus shown that exactly the same condition characterize almost $\F$-recurrent backward shifts operators for general (non necessarily upper) finitely invariant families. 

By combining Proposition \ref{vectores densos tienden a cero} and Theorem \ref{ carac PF uni}, we thus obtain an alternative proof of Theorem 5.1 from \cite{BonGro18}.
\begin{corollary}\label{coro: equiv almost F-rec F-hyp}
   Let $X$ be a Banach space with unconditional basis $\{e_n : n \in \mathbb N_0\}$ and $B$ the unweighted
backward shift operator acting on $X$. Let $\F$ be a uniformly finitely invariant upper hereditary upward family.  Then the following are equivalent.
\begin{enumerate}
    \item [i)] $ B$ is almost $\F$-recurrent;
    \item [ii)] $ B$ is $\F$-hypercyclic.
\end{enumerate}
\end{corollary}
\begin{remark}\rm\label{rem: ejemplo non upper}
It is worth noting that condition $ii)$ in Theorem \ref{ carac PF uni} does not characterize $\F$-hypercyclic backward shifts for general (non upper) finitely invariant families. Indeed, any chaotic operator is almost ${\mathcal Syn}$-recurrent. In particular, there exist almost $\mathcal Syn$-recurrent backward shift operators(which are characterized by Theorem \ref{ carac PF uni}). But, on the other hand, there are no ${\mathcal Syn}$-hypercyclic operators, see  \cite[Proposition 3]{Bes16} (or Corollary \ref{BD>delta}).
\end{remark}

\begin{proof}[Proof of Theorem \ref{ carac PF uni}]
By the continuity of $B$ there is $\beta_p$ such that 
\begin{equation}\label{beta_p}
    \|x\|<\beta_p \text{ then } |x_j|<1/2 \text{ for every }0\le j\leq p.
\end{equation}

$i)\Rightarrow ii).$ Let $p,\varepsilon>0.$ 
 Let $y=\sum_{j=0}^p j e_j$ and $x\in X$
such that $A_{p,\varepsilon}:=\{n: \|B^n(x)-y\|<min\{\f{\varepsilon}{2C},\beta_p\}\}\in \F$, where $C$ is the unconditional basis constant.

We claim that $A_{p,\varepsilon}$ satisfies conditions $(a)$ and $(b)$.

We notice first that if $n\in A_{p,\varepsilon}$ and $0\le j\leq p$, then $\|B^n(x)-y\|<\beta_p$ and hence by \eqref{beta_p}, 
\begin{equation}\label{x_{n+j}}
|x_{n+j}-j|<\f{1}{2} \text{ for every }0\le j\leq p \text{ and } n\in A_{p,\varepsilon}.    
\end{equation}
This implies in particular that $\f{1}{|x_{n+p}|}<2$ for every $n\in A_{p,\varepsilon}$. Since the basis is unconditional and $x\in X$ we obtain that
$$
\sum_{n\in A_{p,\varepsilon}}  e_{n+p}=
\sum_{n\in A_{p,\varepsilon}} \f{1}{x_{n+p}} x_{n+p} e_{n+p}$$
converges in $X$. This shows condition $(a)$.

Let $m\in A_{p,\varepsilon}$. Then, if we denote by $[z]_k$ the $k$-th coordinate of a vector $z$,
$$
\left\|\sum_{n\in A_{p,\varepsilon}, n> m} \f{1}{x_{n+p}} [B^m(x)-y]_{n-m+p} e_{n-m+p}\right\|\leq   2C \|B^m(x)-y\|< \varepsilon.
$$

If $m<n\in A_{p,\varepsilon}$ and $0\le k\leq p,$ then by \eqref{x_{n+j}}, $|x_{n}-x_{m+k}+k|<1.$ Suppose $n=m+k$ for some $k$ such that $1\le k\le p$. Then we would have $k<1,$ which is impossible. Thus $n>m$ implies $n>m+p.$  Therefore,  $[B^m x-y]_{n-m+p}=x_{n+p}$ and thus,
$$\|\sum_{n\in A_{p,\varepsilon}, n> m} e_{n-m+p}\|=\|\sum_{n\in A_{p,\varepsilon}, n> m} \f{1}{x_{n+p}} [B^m(x)-y]_{n-m+p} e_{n-m+p}\|<\varepsilon. $$

$ii)\Rightarrow i).$ We have to show that $B$ is almost $\F$-recurrent. By the denseness of the finite sequences, it is enough to show that for any open ball $U$ centered at a finitely supported vector, there is $x\in U$ such that $N(x,U)\in\F$. Let $\varepsilon>0$ and $p\in\mathbb N$. Consider an open ball $U$ centered at 
 $y=\sum_{j=0}^p y_je_j\in X$  with radius  $\varepsilon>0.$ Let $A_{p,\tilde\varepsilon}\in \mathcal F$
satisfy conditions $(a)$
and $(b)$, where $\tilde \varepsilon:=\frac{\varepsilon}{M(p+1)\|y\|_\infty}$ and  $M := \max_{0\le k\le p}\|B\|^k$.

By  the finite invariance of $\F$ and condition $(a)$, we may assume that 
$$\|\sum_{n\in A_{p,\tilde\varepsilon}} e_{n+p}\|<{\tilde\varepsilon}$$
and that $A_{p,\tilde\varepsilon}\cap [0,p]=\emptyset$.

 Note that $(b)$ implies that  for any $m\in A_{p,\tilde\varepsilon}$ and $0\leq j\leq p$,
\begin{equation}\label{puedo suponer j}
    \Big\|\sum_{n>m,n\in A_{p,\tilde\varepsilon}} e_{n-m+j}\Big\|\leq \|B\|^{p-j}\|\sum_{n>m,n\in A_{p,\tilde\varepsilon}} e_{n-m+p}\|<\frac{\varepsilon\|B\|^{p-j}}{M(p+1)\|y\|_\infty}<\frac{\varepsilon}{C(p+1)\|y\|_\infty}.
\end{equation}

We may also assume that 
\begin{equation}\label{separados}
    n-m>p\text{ for any }n\text{ and }m\text{ in }A_{p,\tilde\varepsilon}\text{ with }n > m.
\end{equation} Indeed, by considering $\varepsilon<{\beta_p}{M(p+1)\|y\|_\infty}$, then $\|\sum_{n>m,n\in A_{p,\tilde\varepsilon}} e_{n-m}\|<\beta_p$.
Then by \eqref{beta_p}, for every $0\le j\le p$, the $j$-th coordinate of $\sum_{n> m, n\in A_{p,\tilde\varepsilon}} e_{n-m}$ is smaller than $\frac12$, and  this implies that $n-m>p$. 

By (a), the vector
 $$
  x:=y+\sum_{j=0}^p y_jB^{p-j}\Big( \sum_{n\in A_{p,\tilde\varepsilon}} e_{n+p}\Big)=y+\sum_{n\in A_{p,\tilde\varepsilon}} \sum_{j=0}^p y_j e_{n+j}\in X. 
 $$
Moreover, $x\in U$. Indeed, 
$$
\|x-y\|\le (p+1)\|y\|_\infty M\Big\| \sum_{n\in A_{p,\tilde\varepsilon}} e_{n+p}\Big\|<\varepsilon.
$$

Notice that by \eqref{separados} we have that if $m,n\in A_{p,\tilde\varepsilon}$ with $m>n$ and $j\leq p$ then $n-m+j\leq 0$. Consequently $B^m(e_{n+j})=0$. This implies that if $m\in A_{p,\tilde\varepsilon}$  
\begin{align*}
    B^m(x)-y=\sum_{n\in A_{p,\tilde\varepsilon},n>m }\sum_{j=0}^p y_j e_{n-m+j}  =\sum_{j=0}^p y_j\sum_{n> m, n\in A_{p,\tilde\varepsilon}} e_{n+j-m}
\end{align*}
and thus, by \eqref{puedo suponer j} 
$$
\|B^m(x)-y\|\le \sum_{j=0}^p|y_j|\Big\|\sum_{n\in A_{p,\tilde\varepsilon},n>m }e_{n-m+j}\Big\| <\varepsilon.
$$
 Therefore $B^mx\in U$ for every $m\in A_{p,\tilde\varepsilon}$. Hence $B$ is almost $\F$-recurrent.
\end{proof}

Applying a conjugacy argument we obtain a characterization for almost $\F$-recurrent weighted backward shifts.
\begin{theorem}\label{ carac PF uni con pesos}
Let $X$ be a Banach space with  basis $\{e_n:\, n\in\mathbb N_0\}$ and let $(w_n)_{n\in\mathbb N}\subset \mathbb R_+$ be such that the weighted backward shift $B_w$ is well defined. Let $\F$ be a finitely invariant hereditary upward family.  Consider the following statements.
\begin{enumerate}
    \item [i)]$B_w$ is almost $\F$-recurrent;
    \item [ii)]for every $\varepsilon>0$ and every $p\in\mathbb N$ there is a set $A_{p,\varepsilon}\in \F$ such that
\begin{enumerate}
\item $\sum_{n\in A_{p,\varepsilon}} \big(w_1\dots w_{n+p}\big)^{-1}e_{n+p}$ converges and
\item for any $m\in A_{p,\varepsilon}$, $$\Big\|\sum_{n>m,n\in A_{p,\varepsilon}} \big(w_1\dots w_{n-m+p}\big)^{-1}e_{n-m+p}\Big\|<\varepsilon.$$
\end{enumerate}
\end{enumerate}
   Then $ii)\Rightarrow i)$. If the basis is unconditional $ii)\Leftrightarrow i)$.
\end{theorem}
\begin{corollary}
Let $X$ be a Banach space with
unconditional basis $\{e_n : n \in \mathbb N_0\}$,
$(w_n)_n$ a bounded weight.
Let $\mathcal F$ be an upper uniformly finite invariant family. The
following are equivalent: \\
$i)$ $B_w$ is almost $\mathcal F$-recurrent;\\
$ii)$ $B_w$ is $\mathcal F$-hypercyclic.
\end{corollary}

Any chaotic operator is almost $\mathcal Syn$-recurrent. In the case of unweighted backward shifts acting on
Banach spaces with unconditional basis we have the following.
\begin{corollary}\label{Almost Syn backward shifts}
Let $X$ be a Banach space with unconditional basis. The backward
shift operator $B$ is almost $\mathcal Syn$-recurrent if and only if
it is chaotic.
\end{corollary}
\begin{proof}
Recall that $B$ is chaotic if and only if $\sum_{n=0}^\infty e_n$ converges, see e.g. \cite[Theorem 4.6]{GroPer11}.
If $B$ is almost $\mathcal Syn$-recurrent then, by Theorem \ref{ carac PF uni}, there is a syndetic set $A$ such that $\sum_{n\in A} e_n$ converges. Since $A$ has bounded gaps, there is $b>0$ such that $\zN\sub \bigcup_{0\le j\leq b} A-j$. 

If $C>0$ denotes the constant of the basis we obtain that $\sum_{n\in \zN_0} e_n$ is convergent because
$$\left\|\sum_{n\ge N} e_n \right\|\leq C(b+1)\left\|\sum_{0\le j\leq b} B^j\l \sum_{n\in A\cap[N,+\infty)} e_n\r\right\|,$$  which tends to 0 as $N\to\infty$.
\end{proof}

\begin{corollary}
    There exists a unilateral weighted backward shift $B_w$ on $c_0$ which is almost $\mathcal{\overline{D}}$-recurrent but not almost $\mathcal{\underline{D}}$-recurrent.
\end{corollary}
\begin{proof}
In \cite[Theorem 13]{BayRuz15} an example of an upper frequently hypercyclic but not frequently hypercyclic weighted backward shift in $c_0$ was given. Since the operator is upper frequently hypercyclic, it is almost $\overline {\mathcal D}$-recurrent. The main argument given by the authors to prove that the operator is not frequently hypercyclic is that for any set of positive lower density $A$, $\lim_{n\in A} \prod_{j=1}^n w_n  \nrightarrow \infty$ (See specifically [step 5.3, Theorem 13]\cite{BayRuz15}). By equivalence ii) a) of Theorem \ref{ carac PF uni con pesos} applied to $c_0$, it follows that the operator is not almost frequently hypercyclic. 
\end{proof}

In \cite{CardeMur22} it was observed that there exists a weighted backward shift in $c_0$ such that it is $
\kAP$-hypercyclic but not upper frequently hypercyclic (see also \cite {Bes16}). In particular, the operator is almost $\kAP$-recurrent and almost $\mathcal {PS}yn$-recurrent. On the other hand, since chaotic operators weighted backward shifts are frequently hypercyclic, the operator can not be chaotic. By Corollary \ref{Almost Syn backward shifts}, it follows that the operator is not almost $\mathcal Syn$-recurrent.

We now consider bilateral shifts. Our main focus will be on the unweighted case, but as usual, all the results can be translated to the weighted case. 

    Let   $X$ be a Banach space with basis $\{e_n:\, n\in\mathbb Z\}$ where the (unweighted) backward shift $B$, defined by $B(e_n)=e_{n-1}$, is continuous. A conjugation argument similar to the unilateral case allows to transfer the results to the weighted shift setting. The bilateral weighted backward shift with weights $(w_n)_{n\in\mathbb Z}$ on $X$ is conjugate to the unweighted backward shift on  $X(v):=\{x=\sum_{n\in\mathbb Z} x_ne_n\,:\,\sum_{n\in\mathbb Z} x_nv_ne_n\in X\},$ with $v_n=(w_{n+1}\dots w_0)$ in case $n<0$,    $v_0=1$ and $v_n=(w_1\dots w_n)^{-1}$ in case $n>0$.
Indeed, the mapping $\phi_v: X(v)\to X$, $(\phi_v(x))_n\mapsto (x_nv_n)_n$ is an isomorphism and it  satisfies that $\phi_v B=B_w\phi_v$.  Then $B:X(v)\to X(v)$ and $B_w:X\to X$ are factors of each other (see again \cite[Section 4.1]{GroPer11} for more details). In particular, $B$ is either $\F$-hypercyclic, $\F$-recurrent or almost $\F$-recurrent if and only if $B_w$ is.

In \cite{Gro19} a characterization of $\F$-hypercyclic bilateral shifts was proved, which involves a sequence of sets $(A_p)_p$ that interact with each other. We next present a characterization of almost $\F$-recurrent bilateral shifts  which is simpler in the sense that it only needs  a single set $A$, and is very similar to the characterization of the unilateral shifts. It should be mentioned that an advantage of the characterization  \cite[Theorem 14]{Gro19} is that it is in some sense symmetric, and this allows the author to deal with the inverse operator.
\begin{theorem}\label{ carac PF bil}
Let $\F$ be a finitely and left-shift invariant hereditary upward family.
Let $\{e_n :\, n\in \mathbb Z\}$ be a   basis of the Banach space $X$ and $B$ a well defined backward shift. Consider the following statements.
\begin{enumerate}
    \item [i)] $B$ is almost $\F$-recurrent;
    \item [ii)] for every $\varepsilon>0$ and every  $p$ there is a set $A=A_{p,\varepsilon}\in\F$ such that
\begin{enumerate}
\item $\sum_{n\in A} e_{n+p}$ converges and,
\item for any $m\in A$,
$$\Big\|\sum_{n\neq m, n\in A} e_{n-m+p}\Big\|<\varepsilon.$$

\end{enumerate}
\end{enumerate}
  Then $ii)\Rightarrow i)$. If the basis is unconditional $ii)\Leftrightarrow i)$.
\end{theorem}
Since recurrent bilateral backward shifts are hypercyclic \cite{BonGroLopPer22JFA}, the above conditions are also equivalent to hypercyclicity plus almost $\F$-recurrence.

We will only sketch the proof of the theorem because it is similar to the proof  Theorem \ref{ carac PF uni}.
\begin{proof}
By the continuity of $B$ there is $\beta_p$ such that $\|x\|<\beta_p$ then $|x_j|<1/2$ for every $j\leq 2p $.

$i)\Rightarrow ii)$ 
Let $y=\sum_{-p\leq j\leq p} je_j$. Thus, there are $x\in X$ and $A\in \F$ such that $\|B^n(x)-y\|<\min\{\f{\varepsilon}{C},\beta_p\}$ for every $n\in A$, where $C$ is the unconditional basis constant.
Then  $A$ will satisfy conditions $(a)$ and $(b)$. 

For $n\in A$, $|j|\le p$,  $|x_{n+j}-j|<\f{1}{2}$ and thus $|\f{1}{x_{n+j}}|<2$ for $j\ne 0$. 
Unconditionality then implies that 
$\sum_{n\in A}  e_{n+p}= \sum_{n\in A} \f{1}{x_{n+p}}x_{n+p} e_{n+p}$
converges and thus we have $(a)$.

For $(b)$, let $m\in A$. Then we see as in the proof of Theorem \ref{ carac PF uni} that $[B^m(x)-y]_{n-m+j}=x_{n+j}$ for every $n\in A$ and $|j|\leq p$.
Then,
\begin{align*}
\Big\|\sum_{n\in A,n\neq m}  e_{n-m+p}\Big\|=     \Big\|\sum_{n\in A,n\neq m} \f{1}{x_{n+p}} [B^m(x)-y]_{n-m+p} e_{n-m+p}\Big\|\leq 2C\|B^m(x)-y\|\leq \varepsilon.
\end{align*}

$ii)\Rightarrow i)$ By the denseness of the finite sequences, it is
enough to show that for any open ball $U$ centered at a finitely supported vector, there is $x \in U$ such that
$N (x, U ) \in\F$. Let $U$ be an open ball of center $y=\sum_{j=-p}^p y_j e_j$ and radius $\varepsilon>0$.

Let $A\in \F$ given by $ii)$ with to respect $\tilde\varepsilon=\frac{\varepsilon}{M(2p+1)\|y\|_\infty }$, where $M:=\max_{0\le k\le p}\|B\|^k$.
 Considering $\varepsilon<\frac{\beta_p}{M^{2p}}$, we may assume that $|n-m|>p$ for any $n\neq m\in A$.

 Let $x=\sum_{n\in A}\sum_{-p}^p y_je_{n+j}\in X$. 
Since $x=\sum_{j=-p}^py_jB^{p-j}(\sum_{n\in A} e_{n+p})$, it is  well defined by $(a)$. And, by $(b)$, if $m\in A$ we have that
\begin{align*}
    \Big\|B^m(x)-y\Big\|&=\Big\|\sum_{n\neq m, n\in A}\sum_{j=-p}^p y_je_{n+j-m}\Big\| =\Big\|\sum_{j=-p}^p y_j\sum_{n\neq m, n\in A} e_{n+j-m}\Big\|\\ 
    &\leq \sum_{j=-p}^p |y_j|\|B^{p-j}\sum_{n\neq m, n\in A} e_{n+p-m}\| <\varepsilon.
\end{align*}
Let $m_1=\min A$. Then, $B^{m_1}(x)\in U$ and for $m\in A,m>m_1$ we have that $B^{m-m_1}(B^{m_1}(x))=B^m(x)\in U$. By shift invariance this shows that $B$ is almost $\F$-recurrent.

 \end{proof}

Via a conjugation, the result for weighted backward shifts reads as follows.
\begin{theorem}\label{ carac PF bil con pesos}
Let $\F$ be a finitely and left-shift invariant hereditary upward family.
Let $\{e_n :\, n\in \mathbb Z\}$ be a   basis of the Banach space $X$ and $B_w$ a well defined weighted backward shift. Consider the following statements.
\begin{enumerate}
    \item [i)] $B_w$ is almost $\F$-recurrent;
    \item [ii)] for every $\varepsilon>0$ and every  $p\in \mathbb N$ there is a set $A=A_{p,\varepsilon}\in\F$ such that
\begin{enumerate}
\item $\sum_{n\in A} v_{n+p}e_{n+p}$ converges and,
\item for any $m\in A$,
$$\Big\|\sum_{n\neq m, n\in A} v_{n-m+p}e_{n-m+p}\Big\|<\varepsilon,$$
where $v_n=(w_{n+1}\dots w_0)$ in case $n<0$,    $v_0=1$ and $v_n=(w_1\dots w_n)^{-1}$ in case $n>0$.
\end{enumerate}
\end{enumerate}
  Then $ii)\Rightarrow i)$. If the basis is unconditional $ii)\Leftrightarrow i)$.
\end{theorem}

\section{Final comments}\label{final comments}
While we were writing the manuscript we learned about the article \cite{GriLop23} by Grivaux and L\'opez-Mart\'inez. There, it is proved that adjoint operators that are almost reiteratively recurrent have an invariant measure with full support (\cite[Theorem 2.3 and Remark 2.5]{GriLop23}), thus being frequently recurrent (\cite[Theorem 3.3]{GriLop23}).
From this result, part (1) of Corollary \ref{coro to PBF implica PF} can immediately be deduced. On the other hand, it is noteworthy to mention that an application of Theorem \ref{PBF implica PF} together with a previous result by Grivaux and Matheron \cite[Proposition 2.11]{GriMat14} provides an alternative proof of the fact that $T$ supports an invariant measure with full support. Recall that a metric space is said to be Polish if it is separable and complete.
\begin{theorem}[Grivaux and Matheron, \cite{GriMat14}]\label{Matheron}
Let $(X,T)$ be a Polish dynamical system. Assume that $X$ is endowed with a Hausdorff topology $\tau$ coarser than the original topology such that every point of $X$ has a neighborhood basis(with respect to the original topology) consisting of $\tau$ compact sets, and that $T$ is continuous with respect to the topology $\tau$. Then the following assertions are equivalent: 
\begin{enumerate}
\item $T$ admits an invariant measure with full support; 
\item for each open set $V\neq \emptyset$, there is an ergodic measure $\mu_V$ for $T$ such that $\mu_V(V)>0$; 
\item $T$ is almost frequently recurrent.
\end{enumerate}
\end{theorem}
It should be mentioned that the concept of frequently recurrence used in \cite{GriMat14} is actually almost frequent recurrence as we define it here.
\begin{theorem}[Grivaux and L\'opez-Mart\'inez \cite{GriLop23}]
Let $X^*$ be a dual separable Banach space and let $T$ be an adjoint operator that is almost reiteratively recurrent. Then $T$ supports an invariant measure with full support.
\end{theorem}
\begin{proof}
Since $T$ is almost reiteratively recurrent and it is an adjoint operator, by Theorem \ref{PBF implica PF}, it follows that it is almost ${\mathcal D}$-recurrent. Theorem \ref{Matheron}, applied to the $w^*$-topology (using that closed balls are $w^*$-compact) then implies that $T$ supports an invariant measure with full support.
\end{proof}

We would like to end the article with some open problems on almost recurrence.

It is of course clear that an $\F$-recurrent operator is almost $\F$-recurrent.
If $\F=\F_\infty$, the family of infinite subsets, then by \cite[Proposition 2.1]{CosManPar14}, almost ${\F_\infty}$-recurrence implies $\F_\infty$-recurrence. The implication almost $\F$-recurrence $\Rightarrow \, \F$-recurrence also holds when $\F=\mathcal{AP}$ is the family of subsets containing arbitrarily long arithmetic progressions (see \cite[Proposition 5.4]{CardeMur22multipleScand}).
If the operator is an adjoint operator then we know of some other instances where this implication holds: when $\F=\kAP$ or $\F=\mathcal {IAP}$ (see Corollary \ref{coro to PBF implica PF}), when $\F=\BD,$ $\F=\overline{\mathcal{D}},$ or $\F=\underline{\mathcal{D}}$ 
(see \cite[Theorem 3.3 and Remark 2.5]{GriLop23}).
We are currently unaware of the answer to the following
question.
\begin{question}\label{question A}
Let $\F$ be an hereditary upward family, does   almost $\F$-recurrence  imply $\F$-recurrence? Does this implication hold for the family of sets with positive lower density or of syndetic sets?
\end{question}

A related problem is whether the analogous of Theorem \ref{PBF implica PF} holds for $\F$-recurrence.

\begin{question}\label{question B}
Let $T$ be an adjoint operator that is $b\F$-recurrent. Is $T$ $\F$-recurrent? 
\end{question} 
 The answer is known to be affirmative for the families   $\F=\mathcal {IAP}$ and $\F=\underline{\mathcal D}$  \cite{CardeMur22,GriLop23}. 
Note that a positive answer to Question \ref{question A} (for adjoint operators) implies a positive answer to Question \ref{question B}.

By Corollary \ref{coro: equiv almost F-rec F-hyp}, for uniformly finitely invariant upper families, any unilateral backward shift operator is $\F$-hypercyclic whenever it is almost $\F$-recurrent and $\F$ is an upper family.  The analogous result is unclear for bilateral backward shifts.

\begin{question}
Let $\F$ be an upper family and let $B$ be an almost $\F$-recurrent bilateral backward shift. Is $B$ $\F$-hypercyclic?
\end{question} 
For $\F=\underline{ \mathcal{D}}$ we could also ask the following (see also \cite[Question 5.3]{BonGroLopPer22JFA}). 
\begin{question}
Let $B$ be an almost frequently recurrent unilateral backward shift. Is $B$ frequently recurrent? Is $B$ frequently hypercyclic?
\end{question}

Theorem \ref{teo:argyros} leads us to the following natural question.
\begin{question}\label{pregunta existe reiterative no existe frequente}
 Does there exist a Banach space that admits a reiteratively hypercyclic
operator but not any frequently hypercyclic operator?
\end{question}



\end{document}